\documentclass[a4paper,draft]{amsproc}
\usepackage{amssymb}
\usepackage{mathtools}
\usepackage{tensor}
\usepackage[hyphens]{url} 

\theoremstyle{plain}
 \newtheorem{thm}{Theorem}[section]
 \newtheorem{prop}{Proposition}[section]
 \newtheorem{lem}{Lemma}[section]
 
\theoremstyle{definition}
 
 \newtheorem{dfn}{Definition}[section]
\theoremstyle{remark}
 
 \numberwithin{equation}{section}

\renewcommand{\leq}{\leqslant}
\renewcommand{\geq}{\geqslant}
\renewcommand{\setminus}{\smallsetminus}
\numberwithin{equation}{section}

\newtheorem*{lemma*}{Lemma}
\newtheorem{claim}{Claim}
\newtheorem*{claim*}{Claim}
\newcommand{\thistheoremname}{}
\newtheorem{genericthm}[thm]{\thistheoremname}
\newenvironment{namedthm}[1]
{\renewcommand{\thistheoremname}{#1}%
	\begin{genericthm}}
	{\end{genericthm}}
\newcommand\numberthis{\addtocounter{equation}{1}\tag{\theequation}}
\newcommand{\A}{\ensuremath{\mathbb{A}}}
\newcommand{\Z}{\ensuremath{\mathbb{Z}}}
\newcommand{\R}{\ensuremath{\mathbb{R}}}
\newcommand{\F}{\ensuremath{F_f}}
\newcommand{\Ss}{\ensuremath{\mathcal{S}}}
\newcommand{\Sp}{\ensuremath{\mathcal{S}^{prim}}}
\newcommand{\M}{\ensuremath{\mathcal{M}}}
\newcommand{\C}{\ensuremath{\mathbb{C}}}
\newcommand{\Q}{\ensuremath{\mathbb{Q}}}

\newcommand{\GL}{\ensuremath{\text{GL}}}
\newcommand{\SL}{\ensuremath{\text{SL}}}
\newcommand{\SG}{S(\Gamma_0(2);-(\frac{1}{4}+\frac{r^2}{4}))} 
\newcommand{\ind}{\text{Ind}}
\newcommand{\HH}{\ensuremath{\mathbb{H}}}
\newcommand{\MG}{\ensuremath{\M(\GL_2(\mathcal{O}),r)}}

\newcommand{\MW}{\ensuremath{\text{MW}}}
\newcommand{\Ms}{\mathcal{M}^*(\GL_2(\mathcal{O}),r)}
\newcommand{\Bg}{\ensuremath{\mathcal{B}_2}}
\newcommand{\BG}{\ensuremath{\mathcal{B}_4}}

\title[Maass space for lifting to GL(2,B)]{Maass space for lifts to GL(2) over a division quaternion algebra}

\author[Wagh]{\bfseries Siddhesh Wagh} 

\address{ 
Department of Mathematics \\ 
University of Oklahoma  \\ 
Norman\\
United States of America}
\email{swagh@math.ou.edu}

\begin{document}

{\begin{flushleft}\baselineskip9pt\scriptsize

\end{flushleft}}
\vspace{18mm} \setcounter{page}{1} \thispagestyle{empty}

\begin{abstract}
Muto, Narita and Pitale construct counterexamples to the Generalized Ramanujan Conjecture for $\text{GL}_2(B)$ over the division quaternion algebra $B$ with discriminant two via a lift from $\text{SL}_2$. In this paper, we try to exactly characterize the image of this lift. The previous methods of Maass, Kohnen or Kojima do not apply here, hence we approach this problem via a combination of classical and representation theory techniques to identify the image. Crucially, we use the Jacquet Langlands correspondence described by Badulescu and Renard to characterize the representations.
\end{abstract}

\maketitle

\section{Introduction} 

One of the fundamental problems in the theory of automorphic forms or
representations is the Ramanujan conjecture. Originally formulated by Ramanujan as estimation for the Fourier coefficients of the weight 12 holomorphic cusp form $\Delta$ over $\SL_2(\Z)$ on the upper half plane $\mathfrak{h}$, the conjecture has been generalized to functions over a broader set of groups in terms of local representations of the associated automorphic forms. To review it, let $\mathcal{G}$ be a reductive algebraic group over a number field $F$, and let $\A := \otimes_{\nu \leq \infty} 'F_\nu$ be the ring of adeles for $F$, where $F_\nu$ denotes the local field at a place $\nu$. Then, one of the old versions of the Ramanujan conjecture can be stated as follows :

\begin{namedthm}\textbf{Conjecture:}
	Let $\pi \simeq \otimes_{\nu \leq \infty}' \pi_\nu$ be an irreducible cuspidal representation of $\mathcal{G}(\A)$, where $\pi_\nu$ denotes the local component of $\pi$ at the place $\nu$. Then $\pi_\nu$ is tempered for every $\nu \leq \infty$.
\end{namedthm}

This naive version of the Ramanujan conjecture is known to be false with the first numerical counter examples being found by Saito and Kurokawa \cite{Kuro}. Adrianov \cite{And1}, Maass \cite{Maass} and Zagier \cite{Zag} showed that the Saito-Kurokawa lift from elliptic cusp forms to holomorphic Siegel cusp forms of degree two always violates the conjecture. Maass found explicit relations between the Fourier coefficients of the holomorphic Siegel cusp forms which characterize the image of the lift (cf. \cite{Maass}). We shall refer to these as the Maass relations and to the image as the Maass space. In \cite{16}, Ikeda generalized the process of Saito-Kurokawa lifts for holomorphic Siegel cusp forms of higher degree. Kohnen and Kojima characterize the Maass space for Ikeda lifts again via a similar process as that of Maass (cf. \cite{17}, \cite{18}). Both these proofs rely crucially on intermediate spaces of Jacobi forms.

While this naive version of the Ramanujan conjecture is strongly believed for the general linear groups, the generalized version is expected only for generic cuspidal representations of quasisplit reductive groups. Muto, Narita and Pitale in \cite{MNP} provide a counterexample to the Ramanujan conjecture for $\GL_2(B)$ over the division quaternion
algebra $B$ with discriminant two. Note that $\GL_2(B)$ is an inner form of the
split group $\GL_4$. Unlike the cases of Saito-Kurokawa and Ikeda, the construction in \cite{MNP} does not involve any intermediate spaces of Jacobi forms. Instead, given Fourier coefficients $c(N)$ of $f \in \SG$ which is an eigenfunction of the Atkin-Lehner involution, they directly define numbers $A(\beta)$ (cf. \eqref{3.3}). Then they show that these $A(\beta)$ are the Fourier coefficients of some $F_f \in \MG$ by using Maass Converse Theorem (cf. Theorem \ref{Pitale}). Here $\MG$ is the space of Maass forms on the 5-dimensional hyperbolic space with respect to $\GL_2(\mathcal{O})$, where $\mathcal{O}$ is the Hurwitz order in $B$ (see Section 2.2 for details). They further show that if $f$ is a Hecke eigenform, then so is $F_f$ and the representation $\Pi_F \simeq \otimes'_{p \leq \infty} \Pi_{F,p}$ of $\GL_2(B_\A)$ corresponding to $F_f$ is a counterexample to the Ramanujan conjecture. They also show that the image of $\Pi_F$ under the global Jacquet-Langlands correspondence is the irreducible constituent of $\ind_{P_{2,2}(\A)}^{\GL_4(\A)} (|\det|_\A^{-1/2} \sigma_f \times |\det|_\A^{1/2} \sigma_f)$, where $\sigma_f$ is the automorphic representation of $\GL_2(\A)$ corresponding to $f$.

The question we want to answer here is the same as the one Maass answered for the Saito-Kurokawa case in \cite{Maass}. More precisely, we want to characterize the image of this lift, possibly in terms of recurrence relations between their Fourier coefficients. We tackle this problem by first noticing that $A(\beta)$ depends only on $K = |\beta|^2,u$ and $n$ when $\beta = \varpi_2^un\beta_0$ as in \eqref{3.2} (cf. \ref{rec}). 
\begin{dfn} 
	Let $\Ms$ be the subspace of $\MG$ consisting of functions $F$ whose Fourier coefficients $A(\beta)$ satisfy:
	\begin{enumerate}
		\item If $\beta = \varpi_2^u n\beta_0 $ as in (\ref{3.2}), then $A(\beta)$ depend only on $K = |\beta|^2$, $u$  and $n$. We shall then write $A(\beta)$ as $A(K,u,n)$.
		
		\item $A(K,u,n)$ satisfy the recurrence relations : 
		\begin{enumerate}
			
			\item $A(K,u,n) = \frac{-3\epsilon}{\sqrt{2}}A(\frac{K}{2},u-1,n) - A(\frac{K}{4},u-2,n)$ for some $\epsilon \in \{\pm 1\}$, 
			
			\item  $A(K,u,n)= \sum_{d \mid n} d \cdot A(\frac{K}{d^2},u,1)$ .
		\end{enumerate}	
	\end{enumerate}
\end{dfn}
Note that there are no intermediate spaces of Jacobi forms. As a result, we cannot just generalize any of the previous proofs of Maass, Kohnen or Kojima to this case. Instead we take a completely different approach. 

We would like to mention that there is a recent work of Pitale-Saha-Schmidt (c.f. \cite{PSS}) which provides a representation theoretic approach to Saito-Kurokawa lifts which does not use Fourier-Jacobi forms. We note however that the paper is only able to show a one way implication without Jacobi forms, where as we wish to prove both ways.

It is easy to see that the Fourier coefficients of $F_f$ satisfy condition (1). To show that $A(K,u,n)$ also satisfy condition $\eqref{2a}$ and $\eqref{2b}$, we use Legendre's three-square theorem to isolate $c(N)$ as follows:
\begin{prop} 
	Let $N = 4^a b$, where $a,b$ are a non-negative integers and $4 \nmid b $. With assumptions as in Theorem \ref{Th1}, we get
	\begin{equation}
	c(-N) = \frac{A(2N,u,1)}{\sqrt{2N}} + \epsilon \frac{A(N,u-1,1)}{\sqrt{N}}
	\end{equation}
	where \[ 
	u = 
	\begin{dcases}
	2a & \text{if }b \equiv 1,3 \mod (4), \\
	2a+1 & \text{if }b \equiv 2 \mod (4). \\ 
	\end{dcases}
	\].
\end{prop}
Now, we manipulate the defining sum of $A(\beta)$ (c.f. \eqref{3.3}) using these $c(N)$ to show that $A(K,u,n)$ indeed satisfy the recurrence relation \eqref{2b}. The relation \eqref{2a} follows from the fact that $F_f$ is a Hecke eigenform at $p=2$. Hence, we get
\begin{thm}
	Let $f \in \SG$ be an Atkin-Lehner eigenform with eigenvalue $\epsilon \in \{\pm 1 \}$ and which is a Hecke eigenform at $p=2$. Then $\F$ obtained in Theorem \ref{Pitale} belongs to the Maass space $\M^*(\GL_2(\mathcal{O}),r)$.
\end{thm}

This allows us to determine a ``necessary" condition for $F \in \MG$ to be a lift. We would like to show a theorem that this is also a ``sufficient" condition. If $F \in \Ms$, we can still isolate $c(N)$ as before and now the question reduces to showing these are the Fourier coefficients of some $f \in \SG$.
 As a first approach one can try to use the Maass converse theorem to show the automorphy of a function $f$ with Fourier coefficients $\{c(N)\}$. The difficulty is that the analytic properties of the Dirichlet series associated with $F$ do not translate into analytic properties of Dirichlet series obtained from $\{c(N)\}$. To approach this problem by representation theory, we first add the condition that $F$ is a Hecke eigenform for all primes $p$ and obtain the following theorem.

\begin{thm}\label{Th1.3}
	Let $F \in \mathcal{M}^*(\GL_2(\mathcal{O},r)$ such that $F$ is a cuspidal Hecke eigenform. Then, there exists $ f \in \SG$, a Hecke eigenform, such that $F = F_f$.
\end{thm}
 We denote by $\Pi_F \simeq \otimes'_{p \leq \infty} \Pi_{F,p}$ the automorphic representation of $\GL_2(B_\A)$ associated with $F$. Let the image of $\Pi_F$ under the global Jacquet Langlands map be $\Pi \simeq \otimes'_{p \leq \infty} \Pi_p$, a representation of $\GL_4(\A)$. For a cuspidal representation $\sigma$ of $\GL_2(\A)$, we denote by $\MW(\sigma,2)$ the Langlands quotient of $\ind_{P_{2,2}(\A)}^{\GL_4(\A)} (|\det|_\A^{1/2} \sigma \times |\det|_\A^{-1/2} \sigma)$, following the notation of Badulescu and Renard from \cite{4}. The strategy of the proof now is to show that $\Pi = \MW(\sigma,2)$ for $\sigma$ an irreducible cuspidal automorphic representation of $\GL_2(\A)$. We show that $\Pi_{F,p}$ is the unique irreducible constituent of some unramified principal series representation $\ind_{\mathcal{B}_4(\Q_p)}^{\GL_4(\Q_p)}(\chi_1 \times \chi_2 \times \chi_3 \times \chi_4)$ where each $\chi_i$ is an unramified character of $\Q_p^\times$ described in the following proposition. 
 \begin{prop}
 	For every odd prime $p$, there is a $ \lambda_{p} \in \C$ such that, up to the action of the Weyl group, $\chi_i$ are given by the formula
 	\begin{align*} 
 	\chi_1(p) = p^{1/2}\frac{\lambda_p + \sqrt{\lambda_p ^2 - 4}}{2} &; \quad \chi_2(p) = p^{1/2}\frac{\lambda_p - \sqrt{\lambda_p ^2 - 4}}{2};\\
 	\chi_3(p) = p^{-1/2}\frac{\lambda_p + \sqrt{\lambda_p ^2 - 4}}{2} &; \quad \chi_4(p) = p^{-1/2}\frac{\lambda_p - \sqrt{\lambda_p ^2 - 4}}{2} \numberthis
 	\end{align*}
 \end{prop}
This is the most crucial result of the paper. The fact that $\chi_i$ are related in this special way and are not arbitrary is an important consequence of the action of the Hecke algebra and recurrence relations from Definition 1.1 \eqref{2b}.
For $p=2$, the structure of the local component $\Pi_{F,2}$ can be obtained from the action of the Hecke algebra and relation \eqref{2a}. The component $\Pi_{F,\infty}$ follows from Section 6.1 of \cite{MNP}. Conditions on the Satake parameters give us that $\Pi$ is indeed of the form $\MW(\sigma,2)$ for some $\sigma$ representation of $\GL_2(\A)$. For an odd prime $p$, let $\chi_{p}$ be the unramified character of $\Q_p^\times$ such that $\chi_{p} (p) = \frac{\lambda_p + \sqrt{\lambda_p ^2 - 4}}{2}$. At the prime $p= \infty$, let $\chi_{\infty}(a) = |a|^s$ where $s = \frac{\sqrt{-1}r}{2}$. For the prime $p=2$, let $\chi$ be an unramified character of $\Q_2^\times$ with $\chi(2) = - \epsilon$ for $\epsilon$ as in condition \eqref{2a} of Definition \ref{rec}.
 \begin{prop}
 	Let $\sigma = \otimes'_{p \leq \infty} \sigma_p$ be the irreducible cuspidal automorphic representation of $\GL_2(\A)$ such that $\Pi = \MW(\sigma,2)$ as above. Then
 	\begin{equation}
 	\sigma_p = 
 	\begin{cases}
 	\ind_{\mathcal{B}_2(\Q_p)}^{\GL_2(\Q_p)}(\chi_{p}\times \chi_{p}^{-1}) \quad &\text{ for odd } p< \infty, \\
 	\chi St_{\GL_2} \quad &\text{ for } p = 2, \\
 	\ind_{\Bg(\R)}^{\GL_2(\R)}(\chi_\infty \times \chi_\infty^{-1}) \quad &\text{ for } p = \infty.
 	\end{cases}
 	\end{equation}
 \end{prop} 
 We then look at the distinguished vector in $\sigma$ to find a function $f$ associated to $\sigma$. As $\sigma_2$ is Steinberg and $\sigma_\infty$ is principal series, we can show that $f \in \SG$ as required. We complete the proof by showing that $c(N)$ are indeed the Fourier coefficients of $f$ implying $F =F_f$. 

To generalize Theorem \ref{Th1.3} to all $F \in \Ms$ we first show that $\MG$ is finite dimensional and has a Hecke eigenbasis of operators that commute with their adjoint. With this, proving that $\Ms$ is a Hecke invariant subspace suffices as this implies that $\Ms$ has a Hecke eigenbasis $F_i \in \Ms$ which are lifts of $f_i \in \SG$. By linearity of the defining condition $\eqref{3.3}$, $F = \sum_i a_i F_i$ would be a lift of $\sum_i a_i f_i$. 

We prove that $\Ms$ is Hecke invariant by showing that for all the Hecke operators $T_i$, the image under their action $T_i(F)$ satisfies the conditions of Definition \ref{rec}. The condition that Fourier coefficients of $T_i(F)$ depend only on $K,u$ and $n$ is obtained by writing the coefficients of $T_i(F)$ in terms of $A(K,u,n)$ the Fourier coefficients of $F$. Since each of these coefficients depends only on $K,u$ and $n$, so do the coefficients of $T_i(F)$. Condition \eqref{2a} is equivalent to $F$ being a Hecke eigenform at prime $p=2$ so it is valid for all $F \in \Ms$. Condition \eqref{2b} is shown by computing the recurrence sum for  $(T_{i,p}F)(K,u,n)$ and showing that it is equal to $\sum_{d \mid n}d (T_{i,p}F)(K/d^2,u,1)$. Hence we get the result:

\begin{thm} 
	The following are equivalent.
	\begin{enumerate}
		\item $F$ is a lift from an Atkin-Lehner eigenform $f \in \SG$ with eigenvalue $\epsilon \in \{\pm 1 \}$ and which is a Hecke eigenform at $p=2$.
		\item $F$ is an element of the space $\Ms$
	\end{enumerate}
\end{thm}

The outline of this paper is as follow. In Section 2, we first introduce the basic notation and the automorphic forms that concern us in this paper. Following that, in Section 3, we present the main results and relevant information from \cite{MNP} which we will build upon. We explicitly describe the Maass space and present the important recurrence relations in Section 4. Then, in Section 5, we present the complete version of the theorem and prove it using representation theory as described above. The final result of the paper and its proof is proved in Section 6. 

I would like to thank my advisor Professor Ameya Pitale for introducing me to this problem and guiding me patiently through it. I would also like to thank Professor Ralf Schmidt for humoring my various questions from time to time. I am also grateful to Dr. Alok Shukla and Manami Roy for lending me a valuable ear during my times of need. Finally, I would like to thank the referee for corrections and advice as well as pointing a flaw in the final section of the original draft. 
\section{Basic Notation}

\subsection{Quaternion Algebra and the 5-dimensional hyperbolic space}
 Following the notation in \cite{MNP}, let $B$ be the definite quaternion algebra over $\mathbb{Q}$ with discriminant 2. In terms of the basis $\{1,i,j,k\}$, $B = \Q +\Q i +\Q j +\Q k$ with $i,j,k$ satisfying $$i^2 =j^2=k^2 =-1, ij = -ji=k.$$
 $\GL_2(B)$ will be the group of elements of $M_2(B)$ whose reduced norms are non-zero. 
 Let $\mathbb{H}= B \otimes_\Q \R$ be the Hamilton quaternion algebra with $x \rightarrow \bar{x}$ the main involution of $\mathbb{H}$. 
 Let $tr(x) = x + \bar{x}$ and $\nu(x) = x \bar{x}$ be the reduced trace and reduced norm of $x \in \mathbb{H}$ respectively with $|x| = \sqrt{\nu(x)}$.

The general linear group $G \coloneqq \GL_2(\mathbb{H})$  admits an Iwasawa decomposition 
$$G =Z^+ NAK,$$
where
\begin{align*}
&Z^+ \coloneqq \{\Big[\begin{matrix}
c & 0 \\
0 & c \\
\end{matrix}\Big]|c \in \R^\times_+\}, 
N \coloneqq \{n(x) = \Big[\begin{matrix}
1 & x \\
0 & 1 \\
\end{matrix}\Big]|x\in \mathbb{H}\}, \\
& A \coloneqq \{a_y = \Big[\begin{matrix}
\sqrt{y} & 0 \\
0 & \sqrt{y}^{-1} \\
\end{matrix}\Big]|y\in \mathbb{R}^+\},
K \coloneqq \{k \in G: \prescript{t}{}{}\bar{k}k = 1_2 \}.
\end{align*}
 The quotient $G/Z^+K$ can be realized as 
 $$\Big\{\Big[\begin{matrix}
 y & x \\
 0 & 1 \\
 \end{matrix}\Big]|y\in \R^\times_+, x \in \mathbb{H}\Big\}.$$
 This gives a realization of the 5-dimensional real hyperbolic space. 
 
 \subsection{Automorphic forms}
 For $\lambda \in \C$ and an arithmetic subgroup $\Gamma \subset \SL_2(\R)$, let $S(\Gamma,\lambda)$ denote the space of Maass cusp forms of weight $0$ on the complex upper half plane $\mathfrak{h}$ whose eigenvalue with respect to the hyperbolic Laplacian is $-\lambda$.
 
For an arithmetic subgroup $\Gamma \subset \GL_2(\mathbb{H})$ and $r \in \C$, let $\mathcal{M}(\Gamma,r)$ denote the space of smooth functions $F$ on $\GL_2(\mathbb{H})$ which satisfy the following conditions : 
	 \begin{enumerate}
	 	\item $\Omega\cdot F = -\frac{1}{2}\Big(\frac{r^2}{4}+ 1\Big)F$, where $\Omega$ is the Casimir operator defined on page 143 in \cite{MNP},
	 	
	 	\item for any $(z,\gamma, g, k) \in Z^+ \times \Gamma \times G \times K$, we have $F(z\gamma g k) = F(g)$,
	 	
	 	\item F is of moderate growth.
	 		
	 \end{enumerate}

For automorphic forms of $\SL_2(\R)$ we will concern ourselves only with the congruence subgroup $\Gamma_0(2) \subset \SL_2(\R)$ of level 2. For the choice of a discrete subgroup of $\GL_2(\HH)$, note that the definite quaternion algebra $B$ has a unique maximal order $\mathcal{O}$ given by: 
$$ \mathcal{O} = \Z + \Z i + \Z j + \Z \frac{1 + i+ j+ij}{2},$$ 
called the Hurwitz order. The discrete subgroup we shall consider in this case will be $\GL_2(\mathcal{O})$.

Let
\begin{equation} \label{Sdef}
\mathcal{S} \coloneqq \Z(1 -ij) + \Z  (-i- ij) + \Z  (-j - ij) + \Z (2ij)
\end{equation}

denote the dual lattice of $\mathcal{O}$ with respect to the bilinear form on $\HH \times \HH$ defined by Re $= \frac{1}{2} tr$. It can be shown that $\mathcal{S} = \varpi_2 \mathcal{O}$ for $\varpi_2 \coloneqq (1+i)$ which is the uniformizer of $B \otimes_\Q \Q_2$.

In terms of $\mathcal{S}$, following Section 2.3 of \cite{MNP}, any $F \in \MG$ has a Fourier expansion of the form 

$$F(n(x)a_y) = u(y) + \sum_{\beta\in S\setminus \{0\}}A(\beta)y^2K_{\sqrt{-1}r} (2\pi|\beta|y)e^{2\pi\sqrt{-1}\text{Re}(\beta x)}.$$

Here $K_\alpha$ is the modified Bessel function, which satisfies the differential equation 
$$y^2\frac{d^2K_\alpha}{dy^2} + y\frac{dK_\alpha}{dy} -(y^2+\alpha^2)K_\alpha = 0.$$

\section{Lifting from $\SG$ to $\MG$}
We first define the set of primitive elements of $\mathcal{S}$, denoted $\mathcal{S}^{prim}$, by 
\begin{equation}
\mathcal{S}^{prim} \coloneqq \{\beta \in \mathcal{S} \setminus \{0\} | \varpi_2 \mid \beta, \varpi_2 ^2 \nmid \beta, n \nmid \beta \text{ for all odd integers n} \}. \label{3.1}
	\end{equation}
Any $\beta \in \mathcal{S}\setminus \{0\}$ can then be uniquely written as 
\begin{equation}
\beta = \varpi_2^u \cdot n \cdot \beta_0, \label{3.2}
\end{equation}
where $u$ is a non-negative integer, $n$ is an odd positive integer and $\beta_0 \in \mathcal{S}^{prim}$.

Let $c(N)$ be the  Fourier coefficients of $f \in S(\Gamma_0(2),-(\frac{1}{4}+\frac{r^2}{4}))$.  Assuming it is an Atkin-Lehner eigenfunction with eigenvalue $\epsilon \in \{\pm 1\}$, set 
\begin{equation}
A(\beta) \coloneqq |\beta|\sum_{t=0}^{u} \sum_{d \mid n}(-\epsilon)^t c\Big(-\frac{|\beta|^2}{2^{t+1}d^2}\Big). \label{3.3}
\end{equation}

With $A(\beta)$ as defined in equation \eqref{3.3}, Muto, Narita and Pitale prove the following theorem.
\begin{thm}[Theorem 4.4 in \cite{MNP}] \label{Pitale}
	Let $f \in S(\Gamma_0(2),−(\frac{1}{4}+\frac{r^2}{4}))$ with Fourier coefficients $c(N)$ and eigenvalue $\epsilon$ of the Atkin-Lehner involution. Define
\begin{equation}
\F(n(x)a_y) \coloneqq \sum_{\beta \setminus \{0\}} A(\beta)y^2K_{\sqrt{-1}r}(2\pi|\beta|y)e^{2\pi
 \sqrt{-1}Re(\beta x)}
\end{equation}
with $\{A(\beta)\}_{\beta \in \Ss \setminus \{0\}}$ as in (\refeq{3.3}). 
Then we have $\F \in \M (\GL_2(\mathcal{O}), r)$ and $\F$ is a cusp form. 
Furthermore, $\F ̸\not \equiv 0$ if $f \not \equiv 0.$ 
\end{thm}

It has been shown in Section 5 of \cite{MNP} that if $f$ is a Hecke eigenform, then $F_f$ is also a Hecke eigenform. For each prime $p \leq \infty$, let $B_p = B \otimes_\Q \Q_p$. For any finite prime $p \neq 2$, we have $\GL_2(B_p) \cong \GL_4(\Q_p)$. Letting $\mathcal{O}_p$ be the p-adic completion of $\mathcal{O}$, we get, for $p \neq 2$, $\mathcal{O}_p \simeq M_2(\Z_p)$ and $\GL_2(\mathcal{O}_p) \simeq \GL_4(\Z_p)$. Set $K_p = \GL_2(\mathcal{O}_p)$ for all $p < \infty$.

According to \cite{Sat}, the Hecke algebra of $\GL_2(B_p)$ with respect to $\GL_2(\mathcal{O}_p) $ is generated by: 
\[
	\begin{cases}
	\{\varphi_1^{\pm 1}, \varphi_2 \} & \text{ if } p=2\\
	\{\phi_1^{\pm 1}, \phi_2,\phi_3, \phi_4\} & \text{ if } p\neq 2.
	\end{cases}
\]
Here $\varphi_1,\varphi_2$ denote the characteristic functions of 
$$ K_2 \begin{bmatrix}
\varpi_2 & 0 \\
0 & \varpi_2 \\
\end{bmatrix}K_2,
K_2 \begin{bmatrix}
\varpi_2 & 0 \\
0 & 1 \\
\end{bmatrix}K_2 $$ 
respectively, whereas $\phi_1,\phi_2,\phi_3,\phi_4$ denote the characteristic functions of $K_p h_i K_p$ where $h_i, 1 \leq i\leq 4$ are
$$\begin{bmatrix}
p & & &\\
 & p & &\\
 & & p &\\
 & & & p
\end{bmatrix},
\begin{bmatrix}
p & & &\\
& p & &\\
& & p &\\
& & & 1
\end{bmatrix},
\begin{bmatrix}
p & & &\\
& p & &\\
& & 1 &\\
& & & 1
\end{bmatrix},
\begin{bmatrix}
p & & &\\
& 1 & &\\
& & 1 &\\
& & & 1
\end{bmatrix}$$
respectively for $p \neq 2$.

We will define the set $ C_p \coloneqq \{\alpha \in \mathcal{O} | \nu(\alpha) = p \}/ \mathcal{O}^\times$.  Proposition 5.8 of \cite{MNP} allows us to explicitly compute the action of the Hecke operators on the Fourier coefficients $A(\beta)$ of $F \in \MG$. 

\begin{prop} \label{prop3.1}
1. Let $p = 2$ and $h =
\begin{bmatrix}
\varpi_2 & 0\\
0 & 1
\end{bmatrix}$. We obtain
$$(K_2hK_2 \cdot F)_\beta = 2(A(\beta \varpi_2^{-1}) + A(\beta \varpi_2)).$$
2. Let $p$ be an odd prime and $\beta \in \mathcal{S}\setminus \{0\}$.\\
(a) When $h_2 =
\begin{bmatrix}
p & & &\\
& p & &\\
& & p &\\
& & & 1
\end{bmatrix}$
$$(K_phK_p \cdot F)_\beta = p( \sum_{\alpha \in C_p} A(\beta \bar{\alpha}^{-1}) + \sum_{\alpha \in C_p} A(\bar{\alpha}\beta)).$$\\
(b) When $h_4 = \begin{bmatrix}
	p & & &\\
	& 1 & &\\
	& & 1 &\\
	& & & 1
\end{bmatrix}$, 
$$(K_phK_p \cdot F)_\beta = p( \sum_{\alpha \in C_p}A(\alpha^{-1} \beta) + \sum_{\alpha \in C_p}A(\beta \alpha)).$$\\
(c) When $h_3 =
\begin{bmatrix}
	p & & &\\
	& p & &\\
	& & 1 &\\
	& & & 1
\end{bmatrix}$, 
$$(K_phK_p \cdot F)_\beta = (p^2 A(p^{-1}\beta ) + p^2A(p\beta) + p \sum_{(\alpha_1,\alpha_2) \in C_p \times C_p} A(\alpha_1^{-1} \beta \alpha_2)).$$
\end{prop}

This action of the Hecke algebra allows us to find the Hecke eigenvalues for $\F$ in terms of the Hecke eigenvalues of $f$. 

\begin{prop} 
\label{Prop3.2}
Let $f \in \SG$ be a Hecke eigenform with eigenvalue $\lambda_p$ for every prime $p$ and Atkin-Lehner eigenvalue $\epsilon$. Then $F = \F$ as defined in Theorem \ref{Pitale} is a Hecke eigenform with eigenvalues $\tensor*[_p]{\mu}{_1}, \tensor*[_p]{\mu}{_2}, \tensor*[_p]{\mu}{_3}, \tensor*[_p]{\mu}{_4}$ for $\phi_1, \phi_2,\phi_3, \phi_4$ respectively at every odd prime $p$ and $_2\mu_1$, $\tensor*[_2]{\mu}{_2}$ for $\varphi_1, \varphi_2$ at $p=2$. They are related as
\begin{equation} \tensor*[_p]{\mu}{_1} = 1 , \tensor*[_p]{\mu}{_2} = \tensor*[_p]{\mu}{_4} = p(p+1)\lambda_p, \tensor*[_p]{\mu}{_3} = p^2\lambda_{p}^{2} + p^3 + p \label{3.5}
\end{equation}
and 
\begin{equation}
\tensor*[_2]{\mu}{_1} = 1, \tensor*[_2]{\mu}{_2} = -3\sqrt{2}\epsilon. \label{3.6}
\end{equation}
\end{prop}

This is proved in Proposition 5.9, 5.11 and 5.12 of \cite{MNP}. Let the representation $\pi_F$ denote the irreducible cuspidal automorphic representation of $\GL_2(B_\A)$ corresponding to $\F$ with $B_\A \coloneqq \otimes_{p \leq \infty}'B_p$. $\pi_F$ is cuspidal as $\F$ is a cusp form and the irreducibility follows from the strong multiplicity-one result for $\GL_2(B_\A)$ (c.f. \cite{5},\cite{4}). Let $\pi_F \simeq \otimes'_p \pi_p$ with $\pi_p$ an irreducible admissible representation of $\GL_2(B_p)$. The local components $\pi_p$ are completely determined by the Hecke eigenvalues given in Proposition \ref{Prop3.2}.

Let $\mathcal{B}_2$ and $\mathcal{B}_4$ denote the group of upper triangular matrices in $\GL_2$ and $\GL_4$ respectively. Then, for $p < \infty$ and odd, 
$\pi_p$ is the unique spherical constituent of the unramified principal series representation $\text{Ind}_{\mathcal{B}_4(\Q_p)}^{\GL_4(\Q_p)}(\chi_1 \times \chi_2 \times \chi_3 \times \chi_4)$
where $\chi_i$ are unramified character of $\Q_p^\times$, given by
\begin{align*}
\chi_1(p) = p^{1/2}\frac{\lambda_p + \sqrt{\lambda_p^2 -4}}{2}, & \chi_2(p) = p^{1/2}\frac{\lambda_p - \sqrt{\lambda_p^2 -4}}{2},\\ 
\chi_3(p) = p^{-1/2}\frac{\lambda_p + \sqrt{\lambda_p^2 -4}}{2}, & \chi_4(p) = p^{-1/2}\frac{\lambda_p - \sqrt{\lambda_p^2 -4}}{2}.
\end{align*}

For $p =2$, $\pi_2$ is is the unique spherical constituent of the unramified
principal series representation $\text{Ind}_{\mathcal{B}_2(\Q_p)}^{\GL_2(B_2)}(\chi_1 \times \chi_2)$ with

$$
\chi_1(\varpi_2) = -\sqrt{2}\epsilon, \chi_2(\varpi_2) = -1/\sqrt{2}\epsilon.
$$

At the prime $p =\infty$, the archimedian component $\pi_\infty$ is isomorphic to the principal series $\text{Ind}^{\GL_2(\mathbb{H})}_{\mathcal{B}_2(\HH)}(\chi_{\pm \sqrt{-1}r})$ with 

$$\chi_s\Big(\begin{pmatrix}
a & *\\
0 & d
\end{pmatrix} \Big) = \nu(ad^{-1})^s.$$

These local representations are explicitly constructed in Section 6 of \cite{MNP}.

\section{Maass space in $\MG$}

We will call the image of the lift in $\M(\GL_2(\mathcal{O}),r)$ as the Maass space. To characterize the functions in the Maass space, we first define the following subspace.

\begin{dfn} \label{rec}
	Let $\Ms$ be the subspace of $\MG$ consisting of functions $F$ whose Fourier coefficients $A(\beta)$ satisfy:
	\begin{enumerate}
		\item If $\beta = \varpi_2^u n\beta_0 $ as in (\ref{3.2}), then $A(\beta)$ depend only on $K = |\beta|^2$, $u$  and $n$. We shall then write $A(\beta)$ as $A(K,u,n)$.
		
		\item $A(K,u,n)$ satisfy the recurrence relations : 
		\begin{enumerate}
		
		\item $A(K,u,n) = \frac{-3\epsilon}{\sqrt{2}}A(\frac{K}{2},u-1,n) - A(\frac{K}{4},u-2,n)$ for some $\epsilon \in \{\pm 1\}$ \label{2a}
		
		\item  $A(K,u,n)= \sum_{d \mid n} d \cdot A(\frac{K}{d^2},u,1)$ \label{2b}
\end{enumerate}	
\end{enumerate}
\end{dfn}
 
We will define $A(K,u,n) = 0$ if $u$ is negative. These recurrence relations are similar to those of Maass in the case of Saito-Kurokawa lifts in \cite{Maass}. 

\begin{thm} \label{Th1}
Let $f \in \SG$ be an Atkin-Lehner eigenform with eigenvalue $\epsilon \in \{\pm 1 \}$ and which is a Hecke eigenform at $p=2$. Then $\F$ obtained in Theorem \ref{Pitale} belongs to the Maass space $\M^*(\GL_2(\mathcal{O}),r)$.
\end{thm}

To prove the theorem, we first need a formula for the Fourier coefficients $c(-N)$ in terms of the Fourier coefficients $A(\beta)$ of $\F$.

\begin{prop} \label{prop1}
	Let $N = 4^a b$, where $a,b$ are a non-negative integers and $4 \nmid b $. With assumptions as in Theorem \ref{Th1}, we get
	\begin{equation}
		c(-N) = \frac{A(2N,u,1)}{\sqrt{2N}} + \epsilon \frac{A(N,u-1,1)}{\sqrt{N}} \label{4.1}
	\end{equation}
	where \[ 
	u = 
	\begin{dcases}
	2a & \text{if }b \equiv 1,3 \mod (4) \\
	2a+1 & \text{if }b \equiv 2 \mod (4) \\ 
	\end{dcases}
	\]
\end{prop}

Note that more than one $\beta$ may have the same $K,u$ and $n$. However, by construction in equation (\ref{3.3}) all such $\beta$ give the same $A(\beta)$. As such, $c(-N)$ is well defined in terms of $\beta$ representatives of $A(K,u,n)$. We will need the following lemmas for the proof of Proposition \ref{prop1}

	\begin{lem} \label{lem1}
Let $\beta = (x +yi+ zj +w ij) \in \mathcal{S}^{prim}$ iff $|\beta|^2 \equiv 2 \mod 4$ and $\gcd(\beta) \coloneqq \gcd(x,y,z,w)=1$.		
	\end{lem}

 \begin{proof}[Proof of Lemma \ref{lem1}]

Simplifying the condition from (\ref{Sdef}), we see that $x+y+z+w \equiv 0 \mod 2$ and therefore $x^2 + y^2 +z^2 +w^2 \equiv 0 \mod 2 \equiv 0, 2 \mod 4$. If $\beta \in \mathcal{O}$ such that $|\beta|^2 \equiv 0 \mod 2$ then by parity conditions $x+y+z+w \equiv 0 \mod 2$ implying $\beta \in \mathcal{S}$. Hence, $\beta \in S$ iff $|\beta|^2 \equiv 0,2 \mod 4$.

Consider an element $\beta_1 \in \mathcal{S}$ with $\gcd(\beta_1)=1$ such that $\beta_1 \notin \mathcal{S}^{prim}$. Then by definition of $\mathcal{S}^{prim}$ in (\ref{3.1}) this means $\beta_1 = \varpi_2^2\beta$ for some $\beta \in \Ss$. Now, $|\varpi_2^2|^2 = 4$ hence, $4| |\beta_1|^2$ implying $|\beta_1|^2 \equiv 0 \mod 4$. 

Conversely, if $|\beta|^2 \equiv 0 \mod 4$ then $|\varpi_2^{-1}\beta|^2 \equiv 0,2 \mod 4$. Since $\varpi_2^{-1} = \frac{1-i}{2}$ it is an easy verification that $\varpi_2^{-1}\beta \in \mathcal{O}$. Then by first part, $\varpi_2^{-1}\beta \in \mathcal{S}$. Therefore, $\beta \in \varpi_2\mathcal{S}$ which is to say $\beta \notin \Sp$.

Therefore, $\beta \in \mathcal{S}$ satisfies $|\beta|^2 \equiv 2 \mod 4$ with $\gcd(\beta) =1$ iff $\beta \notin \varpi_2 S$ and equivalently $\beta \in \mathcal{S}^{prim}$ as required. \end{proof}

For any $N$, the easiest way to identify $\beta$ with $\gcd(\beta)=1$ and $|\beta|^2 = 2N$ is if $w = 1$ with $x^2 + y^2 +z^2 = 2N -1$. By Legendre's three-square Theorem, an odd number $2N-1$ cannot be written as a sum of three squares iff $2N-1 \equiv 7 \mod 8 \Leftrightarrow 2N \equiv 0 \mod 8 \Leftrightarrow N \equiv 0 \mod 4$. However, if $x^2 +y^2+z^2 = 2N-1$ and $x,y,z$ are all odd then $4|(x^2+y^2+z^2+1)$ and hence $\beta \notin \Ss^{prim}$.

\begin{lem} \label{lem2}
If $N \equiv 1,3 \mod 4$, then there is a $\beta \in \Ss^{prim}$ such that
\begin{equation}
c(-N) = \frac{A(\beta)}{\sqrt{2N}}.
\end{equation}
\end{lem}

\begin{proof}[Proof of Lemma \ref{lem2}]
	If $N \equiv 1,3 \mod 4$ then $2N -1 \equiv 1,5 \mod 8$. That means we can find some $x,y$ and $z$ not all odd such that $2N-1 = x^2+y^2+z^2$. Hence, by lemma \ref{lem1}, $\beta = (x + yi + zj + ij)\in \Ss^{prim}$ and $\beta = \varpi_2^0 \cdot 1 \cdot \beta$. Then (\ref{3.3}) becomes
	$$A(\beta) = |\beta| c\Big(\frac{-|\beta|^2}{2}\Big) = \sqrt{2N}c(-N).$$
	Solving for $c(-N)$ gives the required result.
\end{proof}

If $N \equiv 2 \mod 4$ then we can still find $x,y$ and $z$ such that $x^2 +y^2+z^2 +1^2= 2N$. As $|\beta|^2 = 2N \equiv 0 \mod 4$, $\beta \in \Ss$ but $\beta \notin \Sp$ by Lemma \ref{lem1}. However, $\varpi_2^{-2} \beta \notin \Ss$ as $|\varpi_2^{-2} \beta|^2 \equiv 1 \mod 4$. Therefore $\beta = \varpi_2^1 \cdot 1 \cdot \beta_0$ for some $\beta_0 \in \Ss^{prim}$.

If $N \equiv 0 \mod 4$ then such $x,y,z$ do not exist. In that case, write $N = 4^a b$ with $4 \nmid b$. 
Then, we can find $x,y$ and $z$ such that $x^2 +y^2 +z^2 + 1 = 2b$. 
This allows us to create $\beta$ with $|\beta|^2 = 2N$ such that $\beta = \varpi_2^{2a} \cdot 1 \cdot \beta_0$ if $b \equiv 1,3 \mod 4$ and $\beta = \varpi_2^{2a+1} \cdot 1 \cdot \beta_0$ if $b \equiv 2 \mod 4$ with $\beta_0 \in \Sp$.

\begin{lem}\label{lem3}
	If $N \equiv 0,2  \mod 4$, then there is a $ \beta \in \Ss$ such that
	\begin{equation}
c(-N) = \frac{A(\beta)}{\sqrt{2N}} + \epsilon \frac{A( \varpi_2^{-1}\beta)}{\sqrt{N}}.
\end{equation}
\end{lem}
\begin{proof}[Proof of Lemma \ref{lem3}]
	Unlike Lemma \ref{lem2}, in this case $u \neq 0$. Following (\ref{3.3}), we get
	 $$A(\beta) = |\beta|\sum_{t=0}^u (-\epsilon)^tc\left(-\frac{|\beta|^2}{2^{t+1}}\right) = \sqrt{2N}c(-N) +\sqrt{2N}\sum_{t=1}^u (-\epsilon)^tc\left(-\frac{|\beta|^2}{2^{t+1}}\right)$$ 
	But \begin{align*}
	\sum_{t=1}^u (-\epsilon)^t c\left(-\frac{|\beta|^2}{2^{t+1}}\right) 
	&= \frac{-\epsilon}{\sqrt{N}} \sqrt{N}\sum_{t=0}^{u-1} (-\epsilon)^t c\left(-\frac{|\beta|^2 / 2}{2^{t+1}}\right) 
	\\&= \frac{-\epsilon}{\sqrt{N}} \sqrt{N}\sum_{t=0}^{u-1} (-\epsilon)^t c\left(-\frac{|\varpi_2^{-1}\beta|^2 }{2^{t+1}}\right) 
	\\&= \frac{-\epsilon}{\sqrt{N}} A(\varpi_2^{-1}\beta).
	\end{align*}
	
Rearranging the terms gives us the required result.
\end{proof}
 
 \begin{proof}[Proof of Proposition \ref{prop1}]
 Let $N = 4^a b$ with $a$ and $b$ as in the statement. 
 If $b \equiv 1,3 \mod 4$ then by the remark just before Lemma \ref{lem2}, we can find $\beta' \in \Sp$ such that $|\beta'|^2 = 2b$. Then $\beta = \varpi_2^{2a} \beta'$ has $|\beta|^2 = 2N$, $u=2a$ and $n=1$. Hence, in this case $A(\beta) = A(2N, 2a, 1)$. Then $A(\varpi_2^{-1}\beta) = A(N,2a-1,1)$ and the case $b \equiv 1,3 \mod 4$ of the proposition follows from Lemma \ref{lem3}. If $a =0$ then we are in case of Lemma \ref{lem2} with $A(N,-1,1) =0$.
 
 If $b \equiv 2 \mod 4$ then we can find $\beta' = \varpi_2 \cdot 1 \cdot \beta_0$. Taking $\beta = \varpi_2^{2a} \beta'$, we have $A(\beta) = A(2N, 2a+1,1)$ and $A(\varpi_2^{-1}\beta) = A(N,2a,1)$. The case $b \equiv 2 \mod 4$ of the proposition now follows from Lemma \ref{lem3}.
 \end{proof}

Now that we have an explicit formula for the coefficients $c(-N)$ of $f$ in terms of the coefficients $A(K,u,n)$ of $F_f$, we can prove the claim of Theorem \ref{Th1}.

\begin{proof}[Proof of Theorem \ref{Th1}]
	
The Fourier coefficients of $F_f$ are given in terms of the Fourier coefficients of $f$ as in equation (\ref{3.3}) by: 
$$A(\beta) = |\beta|\sum_{t=0}^{u} \sum_{d \mid n}(-\epsilon)^t c\Big(-\frac{|\beta|^2}{2^{t+1}d^2}\Big).$$
The only properties of $\beta$ used in here are $|\beta|, u$ and $n$. Replacing $|\beta|$ with $|\beta|^2$ we can say that the Fourier coefficients $A(\beta)$ of $F_f$ are depend only on $|\beta|^2, u$ and $n$, satisfying the first condition of Definition \ref{rec}. 

To prove equation \eqref{2b}, we use the value of $c(-N)$ from \eqref{4.1} and substituting $A(K,u,n)$ for $A(\beta)$. Doing so, we get:
\begin{align*}
A(K,u,&n)= 
\sqrt{K} \sum_{t=0}^u \sum_{d \mid n} (-\epsilon)^t c\left(-\frac{K}{2^{t+1}d^2}\right) 
\\=& \sqrt{K} \sum_{t=0}^u \sum_{d \mid n} (-\epsilon)^t {\left(\frac{A(K/(2^{t}d^2),u-t,1)}{\sqrt{K/(2^{t}d^2)}} + \epsilon \frac{A(K/(2^{t+1}d^2),u-t-1,1)}{\sqrt{K/(2^{t+1}d^2)}}\right)}
\\=&\sqrt{K} \sum_{d \mid n} \Bigg(\sum_{t=0}^u \Bigg( (-\epsilon)^t	 \frac{A((K/d^2)/2^{t},u-t,1)}{\sqrt{(K/d^2)/2^{t}}} \\
 &- (-\epsilon)^{t+1} \frac{A((K/d^2)/2^{t+1},u-(t+1),1)}{\sqrt{(K/d^2)/2^{t+1}}}\Bigg)\Bigg)
\\=& \sqrt{K} \sum_{d \mid n} \frac{A(K/d^2,u,1)}{\sqrt{K/d^2}}
\\=& \sum_{d \mid n} d A(K/d^2,u,1)
\end{align*}

For equation \eqref{2a}, note that $F_f$ is Hecke eigenform for $p=2$ since we have assumed the same for $f$. Then, by Proposition 5.10 of \cite{MNP}, the Fourier coefficients of the lift $F_f$ satisfy
$$2(A(\beta \varpi_2) + A(\beta \varpi_2^{-1})) = -3\sqrt{2}\epsilon A(\beta)$$ with $A(\beta \varpi_2^{-1}) = 0$ if $u=0$.
Writing it in terms of $K,u$ and $n$, we get 
$$2(A(2K,u+1,n) + A(K/2,u-1,n)) = -3\sqrt{2}\epsilon A(K,u,n)$$ or equivalently

$$A(K,u,n) = \frac{-3\epsilon}{\sqrt{2}}A(\frac{K}{2},u-1,n) - A(\frac{K}{4},u-2,n)$$
for $u \geq 1$ with $A(\frac{K}{4},u-2,n)= 0$ for $u=1$.
\end{proof}
\section{Main Theorem}

Theorem \ref{Th1} shows that if $F \in \MG$ is a lift then $F \in \mathcal{M}^*(\GL_2(\mathcal{O}),r)$. We wish to prove a converse of Theorem \ref{Th1}. We will first show this under the extra hypothesis that $F$ is a Hecke eigenform and later generalize it to all $F \in \Ms$.

\begin{thm}\label{Thm5.1}
	Let $F \in \mathcal{M}^*(\GL_2(\mathcal{O}),r)$ such that $F$ is a cuspidal Hecke eigenform. Then,  there is a $f \in \SG$, a Hecke eigenform, such that $F = F_f$.
\end{thm}

\subsection{Description of the automorphic representation }
If $F$ is a cuspidal Hecke eigenform, let the automorphic representation associated with it as we have introduced after Proposition 3.2 be denoted by $\Pi_F \simeq \otimes'_{p \leq \infty}\Pi_{F,p}$. At every odd prime $p$, the local component $\Pi_{F,p}$ is a spherical representation of $\GL_2(B_p)$ with $B_p = B \otimes \Q_p$. The representation is cuspidal since the Hecke eigenform $F$ is cuspidal. 

For every odd $p < \infty$ we have $\GL_2(B_p) \cong \GL_4(\Q_p)$. From Section 5.2 and 6.1 of \cite{MNP}, we have $\Pi_{F,p}$ is the unique irreducible constituent of some unramified principal series representation $\ind_{\mathcal{B}_4(\Q_p)}^{\GL_4(\Q_p)}(\chi_1 \times \chi_2 \times \chi_3 \times \chi_4)$ where each $\chi_i$ is an unramified character of $\Q_p$. 

Since $F \in \Ms$, the characters $\chi_1,\chi_2,\chi_3,\chi_4$ have a special form. This is proved in the next proposition.

\begin{prop}\label{prop5.1}
For every odd prime $p$, there is a $\lambda_{p} \in \C$ such that, up to the action of the Weyl group, $\chi_i$ are given by the formula
	\begin{align*} 
	\chi_1(p) = p^{1/2}\frac{\lambda_p + \sqrt{\lambda_p ^2 - 4}}{2} &; \quad \chi_2(p) = p^{1/2}\frac{\lambda_p - \sqrt{\lambda_p ^2 - 4}}{2};\\
	\chi_3(p) = p^{-1/2}\frac{\lambda_p + \sqrt{\lambda_p ^2 - 4}}{2} &; \quad \chi_4(p) = p^{-1/2}\frac{\lambda_p - \sqrt{\lambda_p ^2 - 4}}{2} \numberthis \label{5.1}
	\end{align*}
\end{prop}

The proof of the proposition will use \cite[Lemma~5.10]{MNP}. 
\begin{lem}[{\cite[Lemma~5.10]{MNP}}] \label{l51}
	Let $\beta \in \Sp$. Then 
	
	$$\#\{\alpha \in C_p : p| \beta \alpha \} = \#\{ \alpha \in C_p : p|\alpha\beta\}
	= \begin{cases} 
	1 \quad & \text{if } p\mid |\beta|^2, \\
	0 \quad & \text{if } p\nmid |\beta|^2\\
	\end{cases}$$
	In addition, $p^2$ does not divide $\alpha \beta$ or $\beta \alpha$ for any $\alpha \in C_p$
\end{lem} 
Here $ C_p \coloneqq \{\alpha \in \mathcal{O} | \nu(\alpha) = p \}/ \mathcal{O}^\times$ with $\#(C_p) = (p+1)$.
In terms of $A(K,u,n)$, if $A(\beta) = A(K,0,1)$ with $p\nmid K$ then $A(\alpha \beta) = A(\beta \alpha) = A(pK,0,1)$ for every $\alpha \in C_p$. If $p\mid K$ then there are unique $\alpha_1,\alpha_2 \in C_p$ (not necessarily different) such that $A(\alpha_1 \beta) = A(\beta \alpha_2) = A(pK,0,p)$ and $A(\alpha \beta) = A(\beta \alpha) = A(pK,0,1)$ in every other case.

\begin{proof}[Proof of Proposition 5.1]
	It is enough to show that the Hecke eigenvalues $\tensor*[_p]{\mu}{_1}, \tensor*[_p]{\mu}{_2}, \tensor*[_p]{\mu}{_3}, \tensor*[_p]{\mu}{_4}$ for $F$ satisfy the equation (\ref{3.5}). The fact that this is enough follows from the proof of Proposition 6.2 from \cite{MNP}.
	We will follow notation of Proposition \ref{prop3.1} for the Hecke algebra and refer to diagonal matrices given before it by $h_2, h_3$ and $h_4$ respectively. 
	 
	Since the Maass form $F$ is non-zero, at least one of the Fourier coefficients $A(K,u,n)$ is non-zero. This implies, from the Maass conditions of Definition \ref{rec}, that there exists at least one $K$ such that $A(K,0,1) \neq 0$. Let $K = p^nK_0$ where $K_0$ is co-prime to $p$. Then we claim that
	\begin{equation} \label{lambda p}
	\lambda_p = \frac{A(p^{n+1}K_0,0,1) + A(p^{n-1}K_0,0,1)}{ A(p^nK_0,0,1)}
	\end{equation} with $A(p^{n-1}K_0,0,1) = 0$ if $n=0$.
	
	\textbf{Case 1: }$K=p^0K_0$ which is to say $p \nmid K$. Let $\beta$ such that $A(\beta) = A(K,0,1)$. By Lemma \ref{l51}, for every $\alpha \in C_p$, we have $\beta \alpha^{-1} \notin \mathcal{O}$ and $p\nmid \overline{\alpha}\beta$. Therefore, $A(\beta \bar{\alpha}^{-1}) = 0$ and $A(\bar{\alpha}\beta) = A(pK,0,1)$ for every $\alpha \in C_p$.
	
	\begin{align*}
	(K_p h_2 K_p \cdot F)_{\beta} = & p\big( \sum_{\alpha \in C_p} A(\beta \bar{\alpha}^{-1}) + \sum_{\alpha \in C_p} A(\bar{\alpha}\beta)\big)\\
	= & p\big(\sum_{\alpha \in C_p} 0 + \sum_{\alpha \in C_p} A(pK,0,1)\big) \\
	= & p(p+1)A(pK,0,1) .
	\end{align*}
	
	Hence, we get $\tensor*[_p]{\mu}{_2} = p(p+1)\lambda_p$ with $\lambda_{p}$ given in \eqref{lambda p} as required. The same exact argument also proves that $\tensor*[_p]{\mu}{_4} = p(p+1)\lambda_p$.
	
	To show $\tensor*[_p]{\mu}{_3} = p^2\lambda_{p}^{2} + p^3 + p$, we use condition 2(c) of Propositon \ref{prop3.1} and get  
	\begin{align*}
	(K_p h_3 K_p \cdot F)_{\beta} = & \Big(p^2 A(p^{-1}\beta ) + p^2A(p\beta) + p \sum_{(\alpha_1,\alpha_2) \in C_p \times C_p} A(\alpha_1^{-1} \beta \alpha_2)\Big)\\
	= & p^2\cdot 0 + p^2A(p^2K,0,p)+p((p+1)A(K,0,1)) \\
	= & p^2A(p^2K,0,p)+(p^2+p)A(K,0,1)\\
	= & p^2(pA(K,0,1) + A(p^2K,0,1))+p(p+1)A(K,0,1) \numberthis \label{step5.3}\\
	= & p^3A(K,0,1) + pA(K,0,1) + p^2(A(p^2K,0,1) + A(K,0,1)).
	\end{align*}
	The $A(p^{-1}\beta )=0$ since $p\nmid \beta$. By Lemma \ref{l51}, for each $\alpha_2$ there is a unique $\alpha_1$ such that $\alpha_1 ^{-1}\beta \alpha_2 \in \mathcal{O}$. Hence, there are total $(p+1)$ copies of $A(K,0,1)$, one for each $\alpha_2$. 
For \eqref{step5.3}, we are using our recurrence relation \eqref{2b} of Definition \ref{rec} to expand $A(p^2K,0,p) = pA(K,0,1) + A(p^2K,0,1)$. 
	
It now suffices to prove $A(p^2K,0,1) + A(K,0,1) = \lambda_p ^2 A(K,0,1)$ to show $\tensor*[_p]{\mu}{_3} = p^2\lambda_{p}^{2} + p^3 + p$ as in \eqref{3.5}. However, we know $\lambda_p ^2 A(K,0,1) = \lambda_p (\lambda_p A(K,0,1)) = \lambda_p (A(pK,0,1))$ from the argument for $\tensor*[_p]{\mu}{_2}$. Let $\beta'$ be such that $A(\beta') = A(pK,0,1)$. Then, it follows that 
\begin{align*}
p(p+1)\lambda_p A(pK,0,1) = & (K_p h_2 K_p \cdot F)_{\beta'} \\ 
= &  p\big( \sum_{\alpha \in C_p} A(\beta' \bar{\alpha}^{-1}) + \sum_{\alpha \in C_p} A(\bar{\alpha}\beta')\big)\\
= & p(A(K,0,1) + A(p^2K,0,p) +pA(p^2K,0,1)) \numberthis \label{step5.4}\\
= & p(A(K,0,1) + pA(K,0,1)+ A(p^2K,0,1)+pA(p^2K,0,1)) \numberthis \label{step5.5}\\
= & p(p+1)(A(K,0,1)+A(p^2K,0,1)).
\end{align*}
We use Lemma \ref{l51} to expand out the sums to obtain \eqref{step5.4}. Since $p|pK$, there exists a unique $\alpha \in C_p$ such that $\alpha^{-1}\beta' \in \mathcal{O}$ and $A(\beta' \bar{\alpha}^{-1}) = A(K,0,1)$ for that $\alpha$. $A(\beta' \bar{\alpha}^{-1})= 0$ in the other $p$ cases. In the second sum, there exists a unique $\alpha \in C_p$ such that $A(\bar{\alpha}\beta') = A(p^2K,0,p)$. In the other $p$ cases, $A(\bar{\alpha}\beta') = A(p^2K,0,1)$. We use the recurrence relation \eqref{2b} of Definition \ref{rec} again to obtain \eqref{step5.5}. Hence,  $\tensor*[_p]{\mu}{_3} = p^2\lambda_p+p^3+p$ as required, completing the first case.
	
	\textbf{Case 2: }$K = p^nK_0$ with $n>0$. Let $\beta$ be such that $A(\beta) = A(p^nK_0,0,1)$ where $K_0$ is an even number co-prime to $p$.
	\begin{align*}
	(K_p h_2 K_p \cdot F)_{\beta} = & p\big( \sum_{\alpha \in C_p} A(\beta \bar{\alpha}^{-1}) + \sum_{\alpha \in C_p} A(\bar{\alpha}\beta)\big)\\
	= & p(A(p^{n-1}K_0,0,1) + A(p^{n+1}K_0,0,p) + pA(p^{n+1}K_0,0,1)) \\
	= & p(A(p^{n-1}K_0,0,1) + A(p^{n+1}K_0,0,1)  + pA(p^{n-1}K_0,0,1) \numberthis \label{step5.6}\\ 
	&+ pA(p^{n+1}K_0,0,1)) \\
	= & p(p+1)(A(p^{n-1}K_0,0,1) + A(p^{n+1}K_0,0,1)) .
	\end{align*}
	We use Lemma \ref{l51} again to write the sums in terms of $A(K,u,n)$. As $p|p^{n}K_0$ but $p\nmid \beta$, there exists a unique $ \alpha \in C_p$ such that $\beta \alpha^{-1} \in \mathcal{O}$, for which $A(\beta \bar{\alpha}^{-1}) = A(p^{n-1}K_0,0,1)$. As before, $A(\beta \bar{\alpha}^{-1}) = 0$ in all the other $p$ cases. In the second sum, there exists unique $\alpha\in C_p$ such that $A(\bar{\alpha}\beta) = A(p^{n+1}K_0,0,p)$. In the other $p$ cases, $A(\bar{\alpha}\beta) = A(p^{n+1}K_0,0,1)$. We obtain equation \eqref{step5.6} then by using the recurrence relation \eqref{2b} of Definition \ref{rec} to expand $A(p^{n+1}K_0,0,p)$. Hence, we get $\tensor*[_p]{\mu}{_2} = p(p+1)\lambda_p$ with $\lambda_{p}$ given in \eqref{lambda p} as required. Once again, the same exact argument also proves that $\tensor*[_p]{\mu}{_4} = p(p+1)\lambda_p$.
	
To show that $\tensor*[_p]{\mu}{_3} = \lambda_{p}^2p^2 +p^3+p$, we have to consider two subcases: $n=1$ and $n \geq 2$. We will set both cases up and prove them together.
	
	\textbf{Subcase 1:} Letting $n=1$, we get 
	\begin{align*}
	(K_p h_3 K_p \cdot F)_{\beta} = & \Big(p^2 A(p^{-1}\beta ) + p^2A(p\beta) + p \sum_{(\alpha_1,\alpha_2) \in C_p \times C_p} A(\alpha_1^{-1} \beta \alpha_2)\Big)\\
	= & p^2\cdot 0  + p^2A(p^3K_0,0,p)+p((p+1)A(pK_0,0,1) + pA(pK_0,0,1)) \\
	= & p^2(pA(pK_0,0,1) + A(p^3K_0,0,1))+(2p^2+p)A(pK_0,0,1) \numberthis \label{step5.7}\\
	= & p^3A(pK_0,0,1) + pA(pK_0,0,1) + p^2(A(p^3K_0,0,1) \\
	& + 2A(pK_0,0,1)).
	\end{align*}
We will again use Lemma \ref{l51} to simplify the terms in the summation. The first term, $ A(p^{-1}\beta ) =0$ as $p\nmid \beta$. Since $p|pk$, there exists a unique $\alpha'_2 \in C_p$ such that $p|\beta \alpha'_2$. For that $\alpha'_2$, $A(\alpha_1^{-1}\beta \alpha'_2)= A(pK_0,0,1)$ for every $\alpha_1 \in C_p$. In the other $p$ cases of $\alpha_2$'s, there exists a unique $\alpha_1 \in C_p$ such that $\alpha_1 ^{-1 }\beta \alpha_2 \in \mathcal{O}$ and $A(\alpha_1^{-1}\beta \alpha'_2)= A(pK_0,0,1)$. We use the recurrence relation \eqref{2b} of Definition \ref{rec} to expand $A(p^3K_0,0,p)$ to obtain \eqref{step5.7}. 
	
	\textbf{Subcase 2:}  $n \geq 2$
	\begin{align*}
	(K_p h_3 K_p \cdot F)_{\beta} = & \Big(p^2 A(p^{-1}\beta ) + p^2A(p\beta) + p \sum_{(\alpha_1,\alpha_2) \in C_p \times C_p} A(\alpha_1^{-1} \beta \alpha_2)\Big)\\
	= & p^2\cdot0 + p^2A(p^{n+2}K_0,0,p)+p(pA(p^nK_0,0,1)\\ 
	&+ A(p^nK_0,0,p) + pA(p^nK_0,0,1)) \\
	= & p^2(pA(p^nK_0,0,1) + A(p^{n+2}K_0,0,1))+p((2p+1)A(p^nK_0,0,1)\\ &+ pA(p^{n-2}K_0,0,1))\\
	= & p^2(A(p^{n+2}K_0,0,1) + A(p^{n-2}K_0,0,1) + 2A(p^nK_0,0,1))\\ &+(p^3+p)A(p^nK_0,0,1).
	\end{align*}
	Once again, using Lemma \ref{l51} to simplify the summation, we get $A(p^{-1}\beta )=0$ as $p\nmid \beta$. Since $p|p^nK_0$, there exists a unique $\alpha'_2 \in C_p$ such that $p|\beta \alpha'_2$. Since now $p\mid \beta \alpha'_2$, there exists unique $\alpha_1 \in C_p$ such that $(\alpha_1^{-1}\beta \alpha'_2) =A(p^nK_0,0,p)$. In other $p$ cases of $\alpha_1$, $A(\alpha_1^{-1}\beta \alpha'_2) = A(p^nK_0,0,1)$. In the other $p$ cases of $\alpha_2$'s there exists a unique $\alpha_1$ corresponding to each $\alpha_2$ such that $\alpha_1 ^{-1 }\beta \alpha_2 \in \mathcal{O}$ and $A(\alpha_1^{-1}\beta \alpha'_2) = A(p^nK_0,0,1)$. We use recurrence relation \eqref{2b} twice, more precisely, once for the $p^2A(p^{n+2}K_0,0,p)$ and once for $A(p^nK_0,0,p)$. 
	
	Now, to prove $\tensor*[_p]{\mu}{_3} = p^2\lambda_p+p^3+p$ in both the subcases, it suffices to show that $A(p^{n+2}K_0,0,1) + A(p^{n-2}K_0,0,1) + 2A(p^nK_0,0,1) = \lambda_p^2 A(p^nK_0,0,1)$ where $A(p^{n-2}K_0,0,1) = 0$ for the first subcase. For this, we use that $\lambda_p\big(A(p^{n+1}K_0,0,1) + A(p^{n-1}K_0,0,1)\big) = \lambda_p^2 A(p^nK_0,0,1)$ from before. All that is left to show now is that $A(p^{n+2}K_0,0,1) + A(p^{n}K_0,0,1) = \lambda_p A(p^{n+1}K_0,0,1)$ and $A(p^{n}K_0,0,1) + A(p^{n-2}K_0,0,1) = \lambda_p A(p^{n-1}K_0,0,1)$. Both of these are easy to prove and follow from the computation of $(K_p h_2 K_p \cdot F)_\beta$ done at the start of \textbf{Case 2}.
	
	Thus, the Hecke eigenvalues $\tensor*[_p]{\mu}{_1}, \tensor*[_p]{\mu}{_2}, \tensor*[_p]{\mu}{_3}, \tensor*[_p]{\mu}{_4}$ for $F$ satisfy the equation (\ref{3.5}) as required. Rest of the proof follows from the proof of Proposition 6.2 in \cite{MNP}.
\end{proof}

Proposition \ref{prop5.1} gives us the exact structure of $\Pi_{F,p}$ for all odd primes $p$. Next we give a description of $\Pi_{F,2}$ and $\Pi_{F,\infty}$.

\begin{prop} \label{Prop5.2}
a) The local component $\Pi_{F,2}$ is the unique irreducible constituent of the unramified principal series representation $\text{Ind}_{\mathcal{B}_2(B_2)}^{\GL_2(B_2)}(\chi_1 \times \chi_2)$ with $\chi_1 ,\chi_2$ unramified characters of $B^\times$ such that
$$
\chi_1(\varpi_2) = -\sqrt{2}\epsilon, \chi_2(\varpi_2) = -1/\sqrt{2}\epsilon.
$$

b) At the prime $p= \infty$, the archimedean component $\Pi_{F,\infty}$ is isomorphic to the principal series representation $\text{Ind}^{\GL_2(\mathbb{H})}_{\mathcal{B}_2(\HH)}(\chi_{\pm \frac{\sqrt{-1}r}{2}})$ where
$$\chi_s\Big(\begin{pmatrix}
a & *\\
0 & d
\end{pmatrix} \Big) = \nu(ad^{-1})^s$$
\end{prop}

\begin{proof}[Proof of Proposition \ref{Prop5.2}]
	a) For the structure of $\Pi_{F,2}$, it is again enough to show that the Hecke eigenvalues $_2\mu_1$, $\tensor*[_2]{\mu}{_2}$ satisfy the equation (\ref{3.6}). The proof of this is simpler than the odd prime case. From the Maass space condition \eqref{2a} in Definition \ref{rec}, we have 
	$$A(2K,u+1,n) = \frac{-3\epsilon}{\sqrt{2}}A(K,u,n) - A(\frac{K}{2},u-1,n)$$
	 which gives us
	 $$2(A(2K,u+1,n) + A(K/2,u-1,n)) = -3\sqrt{2}\epsilon A(K,u,n).$$
	 Let $\beta \in \Ss$ such that $A(\beta) = A(K,u,n)$. Then, in terms of $\beta$, the above condition can be written as
	$$2(A(\beta \varpi_2) + A(\beta \varpi_2^{-1})) = -3\sqrt{2}\epsilon A(\beta).$$
	Comparing with condition 1 of Proposition 3.1, we get that the Hecke eigenvalue $\tensor*[_2]{\mu}{_2} = -3\sqrt{2}\epsilon$. The rest of the argument follows from Sections 5.2 and 6.1 of \cite{MNP}.
	
b) The proof for the structure of $\Pi_{F,\infty}$ is the same as in Section 6.1 of \cite{MNP} since we still have a Maass form with Casimir eigenvalue $-\frac{1}{2}\left(\frac{r^2}{4} +1\right)$.
\end{proof}

\subsection{Jacquet Langlands correspondence}
Let $B_\A$ denote the adelization of $B$ with $B_p = B \otimes_\Q \Q_p$ as before. 
Badulescu and Renard in \cite{4} give a map $\textbf{G}$ from the automorphic representations on $\GL_2(B_\A)$ to those on $\GL_4(\A)$. Let $\Pi$ denote the image of $\Pi_F$ under $\textbf{G}$.

We will say $\pi = \MW(\sigma,k)$ if a discrete series representation $\pi$ is the unique irreducible quotient of the induced representation $\nu^{(k-1)/2}\sigma \times \nu^{(k-3)/2}\sigma \times \ldots \times \nu^{-(k-1)/2}\sigma$. Here $\sigma$ is cuspidal and $\nu$ is the global character given by product of local characters i.e. absolute value of reduced norm. For $\GL_4$, the possible values of $k$ will be 1,2 or 4. In each of these cases, $\sigma$ will be a cuspidal representation of $\GL_4$, $\GL_2$ and $\GL_1$ respectively over the appropriate group.

By Proposition 18 part b) of \cite{4}, since $\Pi_F$ is cuspidal, its image $\Pi$ is of the form $\MW(\sigma,k_\sigma)$. By Proposition 18 part a) of \cite{4}, $k_\sigma \mid d$ when the dimension of the division algebra is $d^2$. In our case, the division algebra is a quaternion algebra, so $d=2$. Hence, $k_\sigma \mid 2$ implying $k_\sigma = 2$ or $k_\sigma = 1$. The latter condition is same as $\sigma$ being cuspidal.

\begin{prop} \label{prop5.3}
	Let $F \in \Ms$ be a cuspidal Hecke eigenform with $\Pi_F$ the associated representation of $\GL_2(B_\A)$. Then $\textbf{G}(\Pi_F) = \MW(\sigma, 2)$ for some cuspidal automorphic representation $\sigma$ of $\GL_2(\A)$. 
\end{prop}
\begin{proof}
	We will show that $ \textbf{G}(\Pi_F) = \Pi$ is not cuspidal, which is equivalent to showing $k_\sigma \neq 1$. Since $k_\sigma = 1$ or $k_\sigma =2$, this proves the proposition.
	
	For the sake of contradiction, assume $k_\sigma =1$. Therefore, $\Pi$ is a cuspidal automorphic representation of $\GL_4(\A)$. Then, by equation (14) of Sarnak \cite{7}, we get 
	$$\left|\log_p\left(\left|\alpha_i\left(\Pi_p\right)\right|_p\right)\right| \leq \frac{1}{2} - \frac{1}{4^2 +1}.$$
	Here $\alpha_i(\Pi_p)$ denotes the i-th Satake parameter of $\Pi$ at prime $p$, $| \mid_p$ denotes the p-adic valuation and the outer $| \mid$ denotes the standard absolute value. We have $\alpha_i(\Pi_p) = \chi_{i}(p)$ with $\chi_{i}(p)$ as given in \eqref{5.1}. This, in particular, tells us that 
	$$\left|\log_p \left( \left|p^{1/2} \frac{\lambda_p \pm \sqrt{\lambda_p ^2 - 4}}{2}\right|_p \right)\right| \leq \frac{1}{2} - \frac{1}{4^2 +1}$$ and 
	$$\left|\log_p \left( \left|p^{-1/2} \frac{\lambda_p \pm \sqrt{\lambda_p ^2 - 4}}{2}\right|_p \right)\right| \leq \frac{1}{2} - \frac{1}{4^2 +1}.$$
	Therefore, we can write
	\begin{equation} 
	 \left|\log_p \left(\left|p^{\pm 1/2}\right|_p\right) + \log_p \left( \left| \frac{\lambda_p \pm \sqrt{\lambda_p ^2 - 4}}{2}\right|_p \right)\right| \leq \frac{1}{2} - \frac{1}{17}.\label{1/2-}
	 \end{equation}
	
	Let $\log_p \left( \left| \frac{\lambda_p + \sqrt{\lambda_p ^2 - 4}}{2} \right|_p \right) = \alpha^+$ and $\log_p \left(\left|\frac{\lambda_p - \sqrt{\lambda_p ^2 - 4}}{2} \right|_p\right) = \alpha^-$ for convenience of notation. Note that $\alpha^+ + \alpha^- = 0$. Then \eqref{1/2-} implies that
	\begin{align*}
	\left|\frac{1}{2} + \alpha^+ \right| \leq \frac{1}{2} - \frac{1}{17} &&
	\left|\frac{1}{2} + \alpha^- \right| \leq \frac{1}{2} - \frac{1}{17} \\
	\left|\frac{-1}{2} + \alpha^+ \right| \leq \frac{1}{2} - \frac{1}{17} &&
	\left|\frac{-1}{2} + \alpha^- \right| \leq \frac{1}{2} - \frac{1}{17}
	\end{align*}
	In particular, we get that $$-\frac{1}{2} + \frac{1}{17} \leq \frac{1}{2} + \alpha^+  \leq \frac{1}{2} - \frac{1}{17} \quad \text{ and } \quad-\frac{1}{2} + \frac{1}{17} \leq \frac{-1}{2} + \alpha^+ \leq \frac{1}{2} - \frac{1}{17}.$$ 
	Simplifying, we get
	$$ -1 + \frac{1}{17} \leq  \alpha^+  \leq - \frac{1}{17} \text{\quad and \quad } \frac{1}{17} \leq  \alpha^+  \leq 1 - \frac{1}{17}$$
	both of which cannot be simultaneously true. This gives us a contradiction to our starting assumption that $k_\sigma = 1$. Hence, $k_\sigma \neq 1$ which implies $k_\sigma =2 $ as required.
\end{proof}

From Proposition \ref{prop5.3} we obtain an irreducible cuspidal automorphic representation $\sigma$ of $\GL_2(\A)$. We will next describe the local components of $\sigma$ and use that to construct $f \in \SG$, completing the proof of Theorem \ref{Thm5.1}.

\subsection{Description of $\sigma$}

Let $\sigma$ be as from Proposition \ref{prop5.3}, with $\sigma \simeq \otimes'_{p \leq \infty} \sigma_p$. 
For an odd prime $p$, let $\chi_{p}$ be the unramified character of $\Q_p^\times$ such that $\chi_{p} (p) = \frac{\lambda_p + \sqrt{\lambda_p ^2 - 4}}{2}$ for $\lambda_{p}$ as in Proposition \ref{prop5.1}. At the prime $p= \infty$, let $\chi_{\infty}(a) = |a|^s$ where $s = \frac{\sqrt{-1}r}{2}$. For the prime $p=2$, let $\chi$ be an unramified character of $\Q_2^\times$ with $\chi(2) = - \epsilon$ for $\epsilon$ as in condition \eqref{2a} of Definition \ref{rec}.

\begin{prop}\label{Prop5.4}
Let $\sigma \simeq \otimes'_{p \leq \infty} \sigma_p$ be the irreducible cuspidal automorphic representation of $\GL_2(\A)$ from Proposition \ref{prop5.3}. Then
	\begin{equation} \label{5.4}
	\sigma_p = 
	\begin{cases}
	\ind_{\mathcal{B}_2(\Q_p)}^{\GL_2(\Q_p)}(\chi_{p}\times \chi_{p}^{-1}) \quad &\text{ for odd } p< \infty, \\
	\chi St_{\GL_2} \quad &\text{ for } p = 2, \\
	\ind_{\Bg(\R)}^{\GL_2(\R)}(\chi_\infty \times \chi_\infty^{-1}) \quad &\text{ for } p = \infty.
	\end{cases}
	\end{equation}
\end{prop}
\begin{proof} 

The local Jacquet-Langlands map $\textbf{C}$ allows us to explicitly write down $\sigma_p$ at each prime $p$. The local map $\textbf{C}$ is identity at every odd prime $p$ by \cite{4}. At the ramified places, it is given in Theorem 3.2 in \cite{5} for prime $p=2$ and in Section 1.3 of \cite{4} for $p= \infty$. 

Since the $\textbf{C}$ map is identity at every odd prime $p$, we have $\Pi_{F,p} = \Pi_p$ where $\Pi_p$ is the local component of $\Pi$ at prime $p$. Let $P_{2,2}$ denote the 2,2-parabolic subgroup of $\GL_4$. From Proposition \ref{prop5.3} we know that $\Pi_{F,p} = \MW(\sigma_p,2)$. We also know that $\Pi_{F,p}$ is the spherical component of $\ind^{\GL_4(\Q_p)}_{\BG(\Q_p)}(\chi_1 \times \chi_2 \times\chi_3 \times\chi_4)$. We will denote $\ind^{\GL_4(\Q_p)}_{\BG(\Q_p)}(\chi_1 \times \chi_2 \times\chi_3 \times\chi_4)$ by $\ind_{\BG(\Q_p)}^{\GL_4(\Q_p)}(\chi')$ for ease of notation. The reduced norm here is just $\nu = |\det|$.

If $\Pi_{F,p} = \MW(\sigma_p,2)$, then such a $\sigma_p$ is unique (see \cite{13} Section 8). Hence, to show the structure of $\sigma_p$, it is enough to prove the following claim:
\begin{claim}
For every odd prime $p$, $$\ind^{\GL_4(\Q_p)}_{\BG(\Q_p)}(\chi_1 \times \chi_2 \times\chi_3 \times\chi_4) \simeq \ind^{\GL_4(\Q_p)}_{P_{2,2}(\Q_p)}(\nu^{1/2}\sigma_p \times \nu^{-1/2} \sigma_p)$$ 
for 
$\sigma_p = \ind_{\Bg(\Q_p)}^{\GL_2(\Q_p)}(\chi_p \times \chi_p ^{-1})$.
\end{claim}

Using method similar to 6.5 in \cite{3}, define the map 
$$L:\ind_{P_{2,2}(\Q_p)} ^{\GL_4(\Q_p)}( \nu^{1/2}\sigma_p \times \nu^{-1/2} \sigma_p) \rightarrow \ind_{\BG(\Q_p)}^{\GL_4(\Q_p)}(\chi')$$
by 
$$(Lh)(g) := (h(g))(I_2, I_2).$$
	
	Here $h$ is a function in $\ind_{P_{2,2}(\Q_p)} ^{\GL_4(\Q_p)}( \nu^{1/2}\sigma_p \times \nu^{-1/2}\sigma_p)$ and $I_n$ is the identity matrix in $\GL_n(\Q_p)$. 
	We have to show this map is well defined and an isomorphism.
	
	To show that $L$ is well-defined we have to prove that for any $A \in \BG$, $Lh$ satisfies $(Lh)(Ag) = \delta_{\BG}^{1/2}(A)\chi'(A)(Lh)(g)$. Write $A = \begin{bmatrix}
	a & * & * & * \\
	0 & b & * & * \\
	0 & 0 & c & * \\
	0 & 0 & 0 & d \\
	\end{bmatrix}$. \\
	We have 
	\begin{align*}
	\delta_{\BG}^{1/2}(A)&\chi'(A)(Lh(g))=\\ &=|a^3bc^{-1}d^{-3}|^{1/2}\chi_1(a)\chi_2(b)\chi_3(c)\chi_4(d)(Lh(g))\\
	&= |a|^{3/2}|b|^{1/2}|c|^{-1/2}|d|^{-3/2}|a|^{1/2}\chi_p(a)|b|^{1/2}\chi_p ^{-1}(b)\\ &\quad |c|^{-1/2}\chi_p(c)|d|^{-1/2}\chi_p^{-1}(d)(Lh(g))\\
	&=|a|^2 |b| |c|^{-1} |d|^{-2} \chi_{p}(a)\chi_{p}^{-1}(b)\chi_{p}(c)\chi_{p}^{-1}(d)(Lh(g))
	\end{align*}
	
	On the other hand
	\begin{align*}
	(Lh)(Ag) 
	&=\Big|\det \begin{bmatrix}
	a & * \\
	0 & b \\
	\end{bmatrix}\Big| \Big|\det \begin{bmatrix}
	c & * \\
	0 & d \\
	\end{bmatrix}\Big| ^{-1} \nu^{1/2}(\begin{bmatrix}
	a & * \\
	0 & b \\
	\end{bmatrix})
	\\ & \quad \nu^{-1/2}(\begin{bmatrix}
	c & * \\
	0 & d \\
	\end{bmatrix})
	\Big(\sigma_p
	\big((\begin{bmatrix}
	a & * \\
	0 & b \\
	\end{bmatrix}) \otimes \sigma_p(\begin{bmatrix}
	c & * \\
	0 & d \\
	\end{bmatrix})\big)h(g)\Big)(I_2 \otimes I_2)\\
	&= |a||b||c|^{-1}|d|^{-1}|a|^{1/2}|b|^{1/2} |c|^{-1/2}|d|^{-1/2}h(g)(\begin{bmatrix}
	a & * \\
	0 & b \\
	\end{bmatrix} \otimes \begin{bmatrix}
	c & * \\
	0 & d \\
	\end{bmatrix})\\
	&= |a||b||c|^{-1}|d|^{-1}|a|^{1/2}|b|^{1/2}|c|^{-1/2}|d|^{-1/2} \\ & \quad |a|^{1/2}|b|^{-1/2}\chi_p(a)\chi_p ^{-1}(b)|c|^{1/2}|d|^{-1/2}\chi_p(c)\chi_p^{-1}(d)h(g)(I_2 \otimes I_2 )\\
	&=|a|^2 |b| |c|^{-1} |d|^{-2} \chi_{p}(a)\chi_{p}^{-1}(b)\chi_{p}(c)\chi_{p}^{-1}(d)h(g)(I_2 \otimes I_2)\\
	&=|a|^2 |b| |c|^{-1} |d|^{-2} \chi_{p}(a)\chi_{p}^{-1}(b)\chi_{p}(c)\chi_{p}^{-1}(d)(Lh(g))\\
	&=\delta_{\BG}^{1/2}(A)\chi'(A)(Lh(g))
	\end{align*}
	
	To prove injectivity, we look at two functions $h_1$ and $h_2$ in $\ind_{P_{2,2}(\Q_p)} ^{\GL_4(\Q_p)}( \nu^{1/2}\sigma_p \times \nu^{-1/2}\sigma_p)$. 
	By definition, then $Lh_1 = Lh_2$ implies $h_1(g)(I_2,I_2) = h_2(g)(I_2,I_2)$ for every $g$. 
	Applying $( (\nu^{1/2}\sigma_p)(s_1) \times (\nu^{-1/2}\sigma_p)(s_2))$ for $s_1,s_2 \in \GL_2(\Q_p)$ to both sides, we get $h_1(g)(s_1 , s_2) = h_2(g)(s_1 , s_2)$. 
	Therefore, $h_1 = h_2$. 
	
	To show that it is an isomorphism, we construct an inverse map $$\tilde{L}:\ind_{\BG(\Q_p)}^{\GL_4(\Q_p)}(\chi')\rightarrow \ind_{P_{2,2}(\Q_p)} ^{\GL_4(\Q_p)}( \nu^{1/2}\sigma_p \times \nu^{-1/2} \sigma_p) $$ given by $$(\tilde{L}h)(g)(b_1 , b_2) =  h\Big(\begin{bmatrix}
	b_1 & 0 \\
	0 &  b_2\\
	\end{bmatrix}g\Big)$$ for $b_1,b_2 \in \mathcal{B}_2(\Q_p)$. We can verify that it is well defined by similar computation as above and it is easy to see that $L\circ \tilde{L}$ is identity. Hence $L$ is an isomorphism of representation and that proves the claim.

For the case of $p = \infty$, 

\begin{claim}
	At prime $p =\infty$,  
$$|\textbf{C}|(\Pi_{F,\infty}) \simeq \ind^{\GL_4(\R)}_{P_{2,2}(\R)}(\nu^{1/2}\sigma_\infty \times \nu^{-1/2} \sigma_\infty)$$ for $\sigma_\infty = \ind_{\Bg(\R)}^{\GL_2(\R)}(\chi_\infty \times \chi_\infty ^{-1})$ with $\chi_\infty (a) = |a|^s$ and $s = \frac{\sqrt{-1}r}{2}$. 
\end{claim}

For $p=\infty$, note that calculations in Section 6 of \cite{MNP} for the description of $\Pi_\infty$ are for a general element $F \in \MG$ and are independent of any lifting properties. 
Hence, $\Pi_{F,\infty}$ is the irreducible component of $\ind_{\Bg(\HH)}^{\GL_2(\HH)}( \chi'_s \times \chi_s ^{\prime-1})$ with $\chi'_s = \nu ^{\prime s} (x)$ and $s = \frac{\sqrt{-1}r}{2}$. 
Here $\nu'$ denotes the reduced norm of $\mathbb{H}$ at infinity and is equal to the square root of the absolute value. 

Section 1.3 from \cite{5} tells us that Jacquet-Langlands correspondence in this case will be $\ind_{P_{2,2}(\R)}^{\GL_4(\R)}(\xi_s \times \xi_s^{-1})$ where $|\textbf{C}|(\chi'_s) = \xi_s$ and $|\textbf{C}|(\chi_s^{\prime -1}) = \xi_s ^{-1}$. 
Here $\xi_s, \xi_s ^{-1}$ are characters of $\GL_2(\R)$ with $\xi_s = \chi_s \circ \det$ and $\chi_s (a) = |a|^s$ for $s = \frac{\sqrt{-1}r}{2}$. 
Therefore, $\Pi_\infty = \ind_{P_{2,2}(\R)}^{\GL_4(\R)}(\xi_s \times \xi_s^{-1}).$ 
Now, $\ind_{P_{2,2}(\R)}^{\GL_4(\R)}(\xi_s \times \xi_s^{-1})$ is the irreducible quotient of $\ind_{P_{2,2}(\R)}^{\GL_4(\R)}(\tau_s \times \tau_{-s})$ where $\tau_s = \ind_{\Bg(\R)}^{\GL_2(\R)}(| |^{1/2}\chi_s \times ||^{-1/2}\chi_s)$. Therefore, we get
\begin{align*}
\ind_{P_{2,2}(\R)}^{\GL_4(\R)}(\tau_s \times \tau_{-s}) & \simeq \ind_{\BG(\R)}^{\GL_4(\R)}(||^{1/2}\chi_s\times ||^{-1/2}\chi_s\times ||^{1/2}\chi_{-s} \times ||^{-1/2}\chi_{-s}) \\
&\simeq \ind_{\BG(\R)}^{\GL_4(\R)}(||^{1/2}\chi_s\times ||^{1/2}\chi_{-s}\times ||^{-1/2}\chi_{s} \times ||^{-1/2}\chi_{-s}) \\
& \simeq \ind_{P_{2,2}(\R)}^{\GL_4(\R)}(\nu^{1/2}\sigma_s \times \nu^{-1/2}\sigma_s)
\end{align*}
where $\sigma_s = \ind_{\Bg(\R)}^{\GL_2(\R)}(\chi_s \times \chi_{-s})$ and $\nu = |\det|$.

However, we know from global Jaquet Langlands of Section 5.2 that $\Pi_\infty$ is also of the form $MW(\sigma_\infty,2)$ which is the irreducible quotient of $\ind_{P_{2,2}(\R)}^{\GL_4(\R)}(\nu^{1/2}\sigma_\infty \times \nu^{-1/2}\sigma_\infty)$. Hence, by uniqueness, we get $\sigma_\infty = \sigma_s$ with the required form, proving the claim.
	
For the case $p=2$, let $\rho$ and $\rho'$ be unitary representations of $\GL_2(\Q_2)$ and $B_2^\times$ respectively.
At prime $p=2$, Theorem 3.2 from \cite{4} tells us $\textbf{C}(\overline{u}(\rho',k)) = u(\rho,k)$ where
 $$u(\rho,k) = Lg(\Pi_{i=0}^{k-1}\nu^{(k-1)/2 - i}\rho), \quad \overline{u}(\rho',k) = Lg(\Pi_{i=0}^{k-1}\nu'^{(k-1)/2 - i}\rho')$$
 with $Lg$ denoting the unique irreducible quotient. Here, $\nu = |\det|$ and $\nu'$ is the reduced norm. 
 In our case, we have $u(\rho,k) = \Pi_2$ and $\overline{u}(\rho',k) = \Pi_{F,2}$.
 We know $\Pi_{F,2}$, unlike at odd primes $p$, is a representation of $\GL_2(B_2)$. On the other hand have $\Pi_2 =\MW(\sigma_2,2)$, hence $k = 2$. Therefore,  $\Pi_{F,2} =  \text{Ind}_{\mathcal{B}_2(\Q_p)}^{\GL_2(B_2)}(\chi_1 \times \chi_2) = \MW(\rho',2)$. Hence, $\rho'$ is a one dimensional representation of $B^\times$ given by a character $\chi' = \chi \circ \nu'$ for an unramified character $\chi$ of $\Q_2^\times$. Comparing with Proposition 5.2, we get $\chi'(\varpi_2) = -\epsilon$. According to Section 56 of \cite{11} such a character corresponds to twisted Steinberg representation $\chi St$ of $\GL_2(\Q_2)$. Hence, $\sigma_2 = \chi St$ with $\chi(2) = - \epsilon $.

\end{proof}

\subsection{Distinguished vector in $\sigma$}

From Section 5.3, we know that $\sigma_p$ is unramified principal series at every prime $p \neq 2$. Hence, the new vector at every prime $p\neq \{2,\infty\}$ is the unique spherical vector $\psi_p$ stable under $K_p = \GL_2(\Z_p)$. At $p= \infty$, we have the unique weight zero fixed vector $\psi_\infty$ which is stable under $K_\infty = O_2(\R)$. 

At $p=2$ the representation is $\chi St$, an unramified twist of Steinberg by the character $\chi$, as in Proposition \ref{Prop5.4} and hence the conductor is $2$. 
Therefore, the new vector $\psi_2$ is invariant under $$K_2 = \left\{\begin{bmatrix}
a & b \\
c & d \\
\end{bmatrix}\Big| a,d\in \Z_2^\times, b\in \Z_2, c\in 2\Z_2 \right\}.$$

Let $\psi = \otimes'_{p\leq \infty}\psi_p \in V_\sigma$.
It satisfies 
$$ \psi(z\gamma g k) = \psi(g) \quad  \text{ for } \gamma \in \GL_2(\Q), z \in Z(\GL_2(\A)), k \in \Pi_{p \leq \infty} K_p.$$

For $g_\infty = \begin{bmatrix}
a & b\\
c & d
\end{bmatrix} \in \SL_2(\R)$, let $g_\infty(i) = \frac{ai+b}{ci+d} = \tau \in \mathfrak{h}$. Consider the function $f_\psi : \mathfrak{h} \rightarrow \C$ associated to $\psi$ defined as $f_\psi(\tau) = f_\psi(g_\infty(i)) = \psi(g_\infty \otimes'_{p<\infty} 1_p)$ where $1_p$ is the identity of $\GL_2(\Q_p)$. Then, for $\gamma \in \Gamma_0(2)$ we have
\begin{align*}
f_\psi(\gamma(\tau)) &= \psi((\gamma g_\infty)\otimes'_{p<\infty}1_p)\\
&= \psi((\otimes'_{p \leq \infty}\gamma^{-1})((\gamma g_\infty)\otimes'_{p<\infty}1_p)) \qquad \qquad &\because \otimes'_{p \leq \infty} \gamma^{-1} \in \GL_2(\Q)\\
&=\psi(g_\infty \otimes'_{p<\infty}\gamma^{-1})\\
&=\psi((g_\infty \otimes'_{p<\infty}1_p)k) \qquad \qquad & k= (1\otimes'_{p<\infty}\gamma^{-1})\\
&=\psi(g_\infty \otimes'_{p<\infty}1_p)  &\because k\in \Pi_{p \leq \infty}K_p\\
&=f_\psi(\tau)
\end{align*}

Hence $f_\psi$ is invariant under the action of $\Gamma_0(2)$. 
Since the local representation $\sigma_\infty$ at $p=\infty$ associated with the vector $\psi_\infty$  is principal series, the function $f_\psi$ is a Maass form.

Following Lemma 9 from \cite{14} for $n=1$, the map $\psi \rightarrow f_\psi$ is Hecke equivariant. The structure of $\sigma_p$ from Section 5.3 allows us to find the Hecke eigenvalues for $f_\psi$ at all odd prime $p < \infty$. Following Proposition 3.1.2 of \cite{12}, the function $f_\psi$ is an eigenfunction of the Atkin-Lehner involution with eigenvalue $-\chi(2) = \epsilon$ from \eqref{5.4} and Hecke eigenvalue $\lambda_2 = \chi(2) = -\epsilon$. By Proposition 4.6.6 of \cite{2}, the Hecke eigenvalue for odd primes $p$ with $\sigma_p = \ind_{\Bg(\Q_p)}^{\GL_2(\Q_p)}(\chi_p \times \chi_p ^{-1})$ would be $(\chi_{p}(p) + \chi_{p}^{-1}(p)) = \lambda_p$. Note that we are using the action of the Hecke algebra as in (30) of \cite{14} here, hence the lack of $p^{1/2}$.

The eigenvalue for the hyperbolic Laplacian is obtained from the Hecke eigenvalue at infinity as by Proposition 2.5.4 from \cite{2}. Following the notation from Bump \cite{2}, in this case, $s_1 = \frac{\sqrt{-1}r}{2}$ and $s_2 = -\frac{\sqrt{-1}r}{2}$. Hence, $s = \frac{1}{2}(\sqrt{-1}r + 1) = \frac{1+ \sqrt{-1}r}{2} $. Then, the eigenvalue for the Laplacian is given by $s(1-s) = \left(\frac{1}{4} + \frac{r^2}{4}\right)$. Hence, $f_\psi$ belongs to $\SG$ as required. We will do an additional verification that this $f_\psi$ does indeed lift to $F$. This will complete the proof of Theorem \ref{Thm5.1}.

\subsection{Fourier coefficients of $f_\psi$}
Let $N = 4^a b$, where $a,b$ are non-negative integers and $4 \nmid b $. For $F\in \Ms$ we can define a sequence of numbers
\begin{equation}
c(-N) \coloneqq \frac{A(2N,u,1)}{\sqrt{2N}} + \epsilon \frac{A(N,u-1,1)}{\sqrt{N}} \label{CN}
\end{equation}
for every $N>0$, where \[ 
u = 
\begin{dcases}
2a & \text{if }b \equiv 1,3 \mod (4) \\
2a+1 & \text{if }b \equiv 2 \mod (4) \\ 
\end{dcases}
\]
It can be proved that this sequence of numbers $\{c(-N)\}$ satisfy \eqref{3.3} in terms of $K,u$ and $n$ just by reversing the argument in proof of Theorem \ref{Th1}.
\begin{align*}
A(K,u,&n)= \sum_{d \mid n} d A(K/d^2,u,1)
\\=& \sqrt{K} \sum_{d \mid n} \frac{A(K/d^2,u,1)}{\sqrt{K/d^2}}
\\=&\sqrt{K} \sum_{d \mid n} \Bigg(\sum_{t=0}^u \Bigg( (-\epsilon)^t	 \frac{A((K/d^2)/2^{t},u-t,1)}{\sqrt{(K/d^2)/2^{t}}} \\
&- (-\epsilon)^{t+1} \frac{A((K/d^2)/2^{t+1},u-(t+1),1)}{\sqrt{(K/d^2)/2^{t+1}}}\Bigg)\Bigg)
\\=& \sqrt{K} \sum_{t=0}^u \sum_{d \mid n} (-\epsilon)^t {\left(\frac{A(K/(2^{t}d^2),u-t,1)}{\sqrt{K/(2^{t}d^2)}} + \epsilon \frac{A(K/(2^{t+1}d^2),u-t-1,1)}{\sqrt{K/(2^{t+1}d^2)}}\right)}
\\ =&\sqrt{K} \sum_{t=0}^u \sum_{d \mid n} (-\epsilon)^t c\left(-\frac{K}{2^{t+1}d^2}\right) 
\end{align*}
We will show that $f_\psi \in \SG$ lifts to $F$ by showing that these $\{c(N)\}$ are the coefficients of $f$ for all $N<0.$ 
This is to say, if $c_\psi(N)$ are the Fourier coefficients of $f_\psi$, then $c(-N) = c_\psi(-N)$ for all $N>0$. 

For every odd prime $p$, since $f_\psi$ is a Hecke eigenform, $c_\psi(N)$ satisfy 
\begin{equation}
p^\frac{1}{2} c_\psi(pN) + p^{\frac{-1}{2}} c_\psi(N/p)  = \lambda_p c_\psi(N). \label{5.5}
\end{equation}
by equation (5.7) of \cite{MNP}, with $c_\psi(N/p) = 0$ if $p\nmid N$. Rewriting this, we get
\begin{equation}
c_\psi(pN) = p^{-1/2}\lambda_p c_\psi(N) - p^{-1} c_\psi(N/p). \label{5.6}
\end{equation}
Since $f$ is a Hecke eigenform at prime $p=2$, from equation (5.6) of \cite{MNP}, its Fourier coefficients also satisfy 
\begin{equation}
c_\psi(2N) = -\frac{\epsilon}{2}c_\psi(N).	\label{5.7}
\end{equation}

Equations \eqref{5.6} and \eqref{5.7} together allows us to write $c_\psi(-N)$ in terms of $c_\psi(-1)$, $\lambda_p$ and $\epsilon$ for all $N$. This also shows that $c_\psi(-1)$ is in fact non-zero. 

\begin{lem}
The sequence of numbers $\{c(N)\}$ as defined in \eqref{CN} satisfy equations \eqref{5.5} and \eqref{5.7}
\end{lem}

\begin{proof}
To prove $\{c(N)\}$ satisfy \eqref{5.5}, we will use that 
$$\lambda_p = \frac{A(p^{n+1}K_0,u,1) + A(p^{n-1}K_0,u,1)}{ A(p^nK_0,u,1)}.$$
Note that this statement is slightly different than our claim in \eqref{lambda p} since we no longer assume $u=0$. This statement is still true however, since the computation of Fourier coefficients under the action of the Hecke algebra in Proposition \ref{prop3.1} (b) holds true for any general $A(K,u,n)$ and not just with $A(K,0,1)$ as we used for Proposition \ref{prop5.1}. It can also be calculated explicitly via the same argument as in proof of Proposition \ref{prop5.1}. In the computation below, we assume $c(-N/p)=0$ and $A(N/p,u-1,1)=0$ are $0$ if $p\nmid N$.
\begin{align*}
p^\frac{1}{2} c(-pN) + p^\frac{-1}{2} c(-N/p) =&  \left(p^{1/2}\frac{A(2pN,u,1)}{\sqrt{2pN}} + \epsilon p^{1/2} \frac{A(pN,u-1,1)}{\sqrt{pN}} \right) \\
 & + \left( p^{-1/2}\frac{A(2N/p,u,1)}{\sqrt{2N/p}} + \epsilon p^{-1/2} \frac{A(N/p,u-1,1)}{\sqrt{N/p}}\right)\\
=& \frac{A(2pN,u,1)}{\sqrt{2N}} + \epsilon \frac{A(pN,u-1,1)}{\sqrt{N}} \\ & + \frac{A(2N/p,u,1)}{\sqrt{2N}} + \epsilon \frac{A(N/p,u-1,1)}{\sqrt{N}}\\
=& \left( \frac{A(2pN,u,1) + A(2N/p,u,1)}{\sqrt{2N}} \right) \\ 
& + \left( \epsilon \frac{A(pN,u-1,1)+A(N/p,u-1,1) }{\sqrt{N}} \right)\\
=&  \frac{\lambda_p A(2N,u,1)}{\sqrt{2N}} + \epsilon \frac{\lambda_p A(N,u-1,1)}{\sqrt{N}}\\
=& \lambda_p \Big(\frac{A(2N,u,1)}{\sqrt{2N}} + \epsilon \frac{A(N,u-1,1)}{\sqrt{N}}\Big)\\
=& \lambda_p c(-N)
\end{align*}
Thus, $\{c(-N)\}$ satisfy \eqref{5.5}. To show equation \eqref{5.7}, we use the condition \eqref{2a} from Definition \ref{rec}. From $c(-N)$ as in \eqref{CN}, we get
\begin{align*}
c(-2N) =& \frac{A(4N,u+1,1)}{\sqrt{4N}} + \epsilon \frac{A(2N,u,1)}{\sqrt{2N}}\\
=& \left(\frac{-3\epsilon}{\sqrt{2}}\frac{A(2N,u,1)}{\sqrt{4N}} - \frac{A(4N,u-1,1)}{\sqrt{4N}}\right) + \epsilon\frac{A(2N,u,1)}{\sqrt{2N}}\\
=& \frac{-3\epsilon}{2}\frac{A(2N,u,1)}{\sqrt{2N}} + \epsilon\frac{A(2N,u,1)}{\sqrt{2N}} - \frac{A(N,u-1,1)}{2\sqrt{N}}\\
=& \frac{-\epsilon}{2}\frac{A(2N,u,1)}{\sqrt{2N}} - \frac{1}{2}\frac{A(N,u-1,1)}{\sqrt{N}}\\
=& \frac{-\epsilon}{2}c(-N)
\end{align*}

\end{proof}
Since $c(N)$ satisfy \eqref{5.5} and \eqref{5.7}, $c(-1)$ is also not $0$. Then, we can normalize the Fourier coefficients $c_\psi(-N)$ so that $c_\psi(-1)= c(-1)$. Since both $\{c(-N)\}$ and $\{c_\psi(-N)\}$ satisfy \eqref{5.5} and \eqref{5.7}, this implies $c(-N) = c_\psi(-N)$ for all $N$. Therefore, the Fourier coefficients of $f_\psi$ satisfy \eqref{3.3} and hence, it is a Hecke eigenform in $\SG$ whose lift $F_f =F$, completing the proof of Theorem \ref{Thm5.1}.

\section{Final Result}
We would like to generalize the result of Theorem $\ref{Thm5.1}$ to all $F \in \Ms$. We will do so by proving that $\MG$ has a Hecke eigenbasis and showing that the Maass space $\Ms$ is stable under the action of all the Hecke operators given in Proposition $\ref{prop3.1}$. If $\Ms$ is stable, then it has a Hecke eigenbasis $\{F_i\}$ which are lifts of some $\{f_i\}$ for $f_i \in \SG$ by Theorem \ref{Thm5.1}. Then by linearity of the defining condition $\eqref{3.3}$, $F = \sum_i a_i F_i$ would be a lift of $\sum_i a_i f_i \in \SG$. Let $\Gamma \subset \GL_2(\mathcal{O})$ be a subgroup of finite index. For Maass forms $F,G$ over $\mathcal{M}(\Gamma,r)$ with one of them cuspidal, we can define their Petersson inner product by 
\begin{equation}
\langle F,G \rangle = \frac{1}{Vol(\Gamma\backslash\GL_2(\mathcal{O}))} \int_{\Gamma\backslash \GL_2(\HH)/Z^+K}F(g)\overline{G(g)}dg \label{Pet}
\end{equation}
where the Haar measure $dg$ is given by $\frac{dxdy}{y^2}$ when $g = \begin{bmatrix}
y & x \\
0 & 1 \\
\end{bmatrix}$ as in Section 2.1.

\begin{prop}
	$\MG$ has a basis of forms that are simultaneous eigenvectors of the Hecke algebra $\otimes \mathcal{H}(G_p,K_p)$ and the subspace of cusp forms has an orthogonal basis of Hecke eigenfunctions with respect to the Petersson inner product on the 5-dimensional hyperbolic as in \eqref{Pet}

\end{prop}
\begin{proof}
	The Hecke algebra acting on $\MG$ is $\otimes \mathcal{H}(G_p,K_p)$ as in Section 5.2 of \cite{MNP}. By Theorem 6 from Section 8 of \cite{Sat}, the algebra $\mathcal{H}(G_p,K_p)$ is commutative for every $p$. 
	
	By Theorem 1 on page 8 of \cite{HC}, dim$(\MG) < \infty$. This means $\MG$ is a finite dimensional vector space with a commutative algebra $\otimes \mathcal{H}(G_p,K_p)$ of operators acting on it. The final step in the proof is to show that the operators commute with their adjoint with respect to the Petersson inner product.
	
	Let $g$ be such that $K_pgK_p$ is one of the generators of the Hecke algebra $\mathcal{H}(G_p,K_p)$ which according to Proposition \ref{prop3.1} are $h_1,h_2,h_3$ and $h_4$ for odd $p$ and 
	$\begin{bmatrix}
		\varpi_2 & 0\\
		0 & \varpi_2
	\end{bmatrix}$ and $\begin{bmatrix}
		\varpi_2 & 0\\
		0 & 1
	\end{bmatrix}$ for $p=2$. Let $K_pgK_p = \sqcup_i K_pg_i = \sqcup_i g_iK_p$. Then $K_pg^{-1}K_p = \sqcup_i K_p g_i^{-1}$ and $K_pg^{-1}zK_p = \sqcup_i K_p g_i^{-1}z$ for $z\in Z$ an element of the center. Let $F,G$ be cusp forms in $\MG$. To find the adjoint operator under the Petersson inner product of \eqref{Pet}, we can look at
	
\begin{align*}
 \langle T(g)F,G\rangle = \sum_i \langle F|_{g_i},G\rangle = & \sum_i \langle F|_{g_i},G|_z\rangle \\
  = & \sum_i \langle F|_{g_i},G|_z|_{g_i^{-1}}|_{g_i}\rangle\\
  = & \sum_i \langle F,G|_{zg_i^{-1}}\rangle\\
  = & \sum_i \langle F,T(zg^{-1})G\rangle.
\end{align*}
Hence, to show that the Hecke operators commute with their adjoint, it suffices to show that $T(zg^{-1})$ is a generator up to a element of the center. Note that while $F,G$ are cusp forms with respect to $\GL_2(\mathcal{O})$, the forms $G|_{zg^{-1}}$ might be modular only with respect to some smaller subgroup $\Gamma$ hence we define Petersson inner product for general subgroup $\Gamma$ rather than just $\GL_2(\mathcal{O})$.  

Now, taking $z = \begin{bmatrix}
\varpi_2 & 0\\
0 & \varpi_2
\end{bmatrix}$ and the Weyl group element $w = \begin{bmatrix}
0 & 1\\
1 & 0
\end{bmatrix} \in K_2$, we can see that $K_2gK_2 = K_2wzg^{-1}wK_2$. Similarly, taking $z = \begin{bsmallmatrix}
p & & & \\
 & p & & \\
 & & p & \\
 & & & p 
\end{bsmallmatrix}$ and $w = \begin{bsmallmatrix}
 & & & 1\\
 & & 1 & \\
 & 1 & & \\
 1 & & & 
\end{bsmallmatrix} \in K_p$, we can show that $K_ph_2K_p = K_pwzh_4^{-1}wK_p$, $K_ph_3K_p = K_pwzh_3^{-1}wK_p$ and $K_ph_4K_p = K_pwzh_2^{-1}wK_p$. Letting $T_{p,i}$ denote the Hecke operator of $K_ph_iK_p$, this shows 
$$T^*_{2,1}=T_{2,1}, \quad T^*_{2,2}=T_{2,2}, $$ and
$$T^*_{p,1}=T_{p,1}, \quad T^*_{p,2}=T_{p,4}, \quad T^*_{p,3}=T_{p,3}, \quad T^*_{p,4}=T_{p,2}$$
for every prime odd $p>2$. 
\end{proof}
\begin{thm} 
	The following are equivalent.
	\begin{enumerate}
		\item $F$ is a lift from an Atkin-Lehner eigenform $f \in \SG$ with eigenvalue $\epsilon \in \{\pm 1 \}$ and which is a Hecke eigenform at $p=2$.
		\item $F$ is an element of the space $\Ms$
	\end{enumerate}
\end{thm}

\begin{proof}
	As mentioned before, it is enough to show that $\Ms$ is stable under the action of the Hecke Algebra. We will prove that for any Hecke operator $T_{p,i}$ and any $F \in \Ms$, the image of the action $T_{p,i}F \in \Ms$ by showing $T_{p,i}F$ satisfies all the conditions of Definition \ref{rec}. 
	
	Condition 1 of Definition \ref{rec} follows from the fact that we can write the coefficients $A'(\beta)$ of $T_{p,i}F$ in terms of $A(K,u,n)$ using Proposition \ref{prop3.1} following argument similar to proof of Proposition \ref{prop5.1}.
	We showed in the proof of Proposition \ref{Prop5.2} that condition $\ref{2a}$ is actually equivalent to $F$ being a Hecke eigenform at prime $p=2$. Hence, $T_{2,1}F = -(3\sqrt{2}\epsilon) F$ for all $F \in \Ms$. 
	Checking the recurrence relations for $T_{p,i}F$ where $p$ is an odd prime requires computation. We will present the most complicated case for $T_{p,2} = h_3$ below with the assumption that $A(\beta) = A(p^mK,u,p^ln)$ where $p\nmid Kn$ and $m>2l+1$. Computation for other cases and other Hecke operators follow similarly. For ease of notation, we will refer to the Fourier coefficients of $T_{p,i}F$ as $T_{p,i}F(K,u,n)$. We will use the recurrence relation as 
	$$A(p^mK,u,p^ln)=\sum_{i=0}^{l}p^iA(p^{m-2i}K,u,n)$$
	 and show that 
	$$T_{p,2}F(p^mK,u,p^ln) = \sum_{i=0}^{l}p^iT_{p,2}F(p^{m-2i}K,u,n).$$ 
	Note, following similar argument as in proof of Proposition \ref{prop5.1}, we have
\begin{align*}
T_{p,2}F(p^mK,u,n) = & p^2(0)+p^2A(p^{m+2}K,u,pn) \\+&p(pA(p^mK,u,n)+pA(p^mK,u,n)+A(p^mK,u,pn))\\
= &p^2A(p^{m+2}K,u,pn)+2p^2A(p^mK,u,n)+pA(p^mK,u,pn).
\end{align*}
Hence,
\begin{align*}
\sum_{i=0}^l p^i T_{p,2}&F(p^{m-2i}K,u,n) =\\ =& \sum_{i=0}^l p^i (p^2A(p^{m+2-2i}K,u,pn) + 2p^2A(p^{m-2i}K,u,n) + pA(p^{m-2i}K,u,pn))\\
=&\sum_{i=0}^{l} p^i(p^2(A(p^{m+2-2i}K,u,n)+pA(p^{m-2i}K,u,n)) + 2p^2A(p^{m-2i}K,u,n) \\ 
+&p(A(p^{m-2i}K,u,n)+pA(p^{n-2-2i}K,u,n)))\\
=& \sum_{i=0}^{l}p^i(p^2A(p^{m-2-2i}K,u,n) +p^2A(p^{m+2-2i}K,u,n) \\ +& (p^3+2p^2+p)A(p^{m-2i}K,u,n)
\end{align*}

We compare that with
\begin{align*}
T_{p,2}F(p^mK,u,p^ln) = & p^2A(p^{m-2}K,u,p^{l-1}n)+p^2A(p^{m+2}K,u,p^{l+1}n)\\
+&  p(p^2A(p^mK,u,p^{l-1}n) + 2pA(p^mK,u,p^ln)\\ 
+& A(p^mK,u,p^{l+1}n))\\
= & p^2 \sum_{i=0}^{l-1}p^iA(p^{m-2-2i}K,u,n)+p^2\sum_{i=0}^{l+1}p^iA(p^{m+2-2i}K,u,n)\\
+& p^3\sum_{i=0}^{l-1}p^iA(p^{m-2i}K,u,n)+2p^2\sum_{i=0}^{l}p^iA(p^mK,u,n)\\
+& p\sum_{i=0}^{l+1}p^iA(p^mK,u,n)\\
= & p^2 \sum_{i=0}^{l-1}p^iA(p^{m-2-2i}K,u,n)+p^2\sum_{i=0}^{l}p^iA(p^{m+2-2i}K,u,n) \\ +&p^2p^{l+1}A(p^{m-2l}K,u,n) +p^3\sum_{i=0}^{l-1}p^iA(p^{m-2i}K,u,n)\\ +&2p^2\sum_{i=0}^{l}p^iA(p^mK,u,n)+ p\sum_{i=0}^{l}p^iA(p^mK,u,n)\\  
+& pp^{l+1}A(p^{m-2-2l}K,u,n).
\end{align*}

Rearranging the sum, we get 
\begin{align*}
T_{p,2}F(p^mK,u,p^ln)= & \sum_{i=0}^{l}p^i(p^2A(p^{m-2-2i}K,u,n) +\sum_{i=0}^{l}p^ip^2A(p^{m+2-2i}K,u,n) \\ 
+& \sum_{i=0}^{l}p^i(p^3+2p^2+p)A(p^{m-2i}K,u,n)\\
=&\sum_{i=0}^l p^i T_{p,2}F(p^{m-2i}K,u,n) 
\end{align*}
as required.
\end{proof}
\bibliographystyle{amsplain}

\end{document}